\documentclass{amsart}
\author{Dilip Raghavan}
\address{Department of Mathematics\\
National University of Singapore\\
Singapore 119076}
\email{raghavan@math.nus.edu.sg}
\urladdr{http://www.math.toronto.edu/raghavan}
\author{Stevo Todorcevic}
\thanks{Second author partially supported by NSERC}
\address{Dpartment of Mathematics, University of Toronto, Toronto, Canada, M5S 2E4}
\address{Institut de Math\'{e}matique de Jussieu, UMR 7586, Case 247, 4 place Jussieu, 75252 Paris Cedex, France}
\email{stevo@math.toronto.edu, todorcevic@math.jussieu.fr}
\date{\today}
\title{Combinatorial dichotomies and cardinal invariants}
\usepackage{amssymb, amsmath, amsthm, mathrsfs, enumerate, amsfonts, latexsym,
bbm}
\def\polhk#1{\setbox0=\hbox{#1}{\ooalign{\hidewidth
    \lower1.5ex\hbox{`}\hidewidth\crcr\unhbox0}}}
\newtheorem{Theorem}{Theorem}
\newtheorem{Claim}[Theorem]{Claim}
\newtheorem{Lemma}[Theorem]{Lemma}
\newtheorem{Cor}[Theorem]{Corollary}
\newtheorem{conj}[Theorem]{Conjecture}
\newtheorem*{samp}{Prototypical Theorem}
\newtheorem*{gq}{General Problem 1}
\newtheorem*{gqq}{General Problem 2}

\newtheorem{prob}[Theorem]{Problem}
\theoremstyle{definition}
\newtheorem{Def}[Theorem]{Definition}

\theoremstyle{remark}

\newcommand{\forces}{\Vdash}
\newcommand{\restrict}{\upharpoonright}
\newcommand{\forallbutfin}{{\forall}^{\infty}}
\newcommand{\existsinf}{{\exists}^{\infty}}
\renewcommand{\c}{\mathfrak{c}}
\renewcommand{\b}{\mathfrak{b}}
\renewcommand{\d}{{\mathfrak{d}}}

\newcommand{\x}{{\mathfrak{x}}}

\newcommand{\p}{{\mathfrak{p}}}

\renewcommand{\[}{\left[}
\renewcommand{\]}{\right]}
\renewcommand{\P}{\mathbb{P}}

\newcommand{\Q}{\mathbb{Q}}
\newcommand{\R}{\mathbb{R}}
\newcommand{\lc}{\left|}
\newcommand{\rc}{\right|}

\newcommand\ZFC{\mathrm{ZFC}}

\newcommand\MA{\mathrm{MA}}
\newcommand\PFA{\mathrm{PFA}}
\newcommand\PID{\mathrm{PID}}
\newcommand\MM{\mathrm{MM}}
\newcommand\RC{\mathrm{RC}}

\newcommand\CH{\mathrm{CH}}
\newcommand\CD{\mathrm{CD}}

\newcommand{\BS}{{\omega}^{\omega}}

\DeclareMathOperator{\pred}{pred}
\DeclareMathOperator{\cone}{cone}
\DeclareMathOperator{\cl}{cl}
\DeclareMathOperator{\tr}{tr}

\DeclareMathOperator{\otp}{otp}
\DeclareMathOperator{\cov}{cov}
\DeclareMathOperator{\add}{add}
\DeclareMathOperator{\cof}{cof}

\DeclareMathOperator{\dom}{dom}

\DeclareMathOperator{\succc}{succ}

\DeclareMathOperator{\hgt}{ht}

\newcommand{\Pset}{\mathcal{P}}
\newcommand{\M}{\mathcal{M}}
\newcommand{\N}{\mathcal{N}}

\newcommand{\G}{{\mathscr{G}}}
\newcommand{\GG}{{\mathcal{G}}}
\newcommand{\HH}{{\mathcal{H}}}

\newcommand{\U}{{\mathcal{U}}}

\newcommand{\cube}{{\[\omega\]}^{\omega}}

\newcommand{\I}{{\mathcal{I}}}
\newcommand{\J}{{\mathcal{J}}}
\newcommand{\LLL}{{\mathcal{L}}}

\newcommand{\F}{{\mathcal{F}}}

\newcommand{\V}{{\mathbf{V}}}
\newcommand{\VG}{{{\mathbf{V}}[G]}}

\newcommand{\Sa}{\mathbb{S}}
\newcommand{\wo}{\blacktriangleleft}
\begin{document}
\begin{abstract}
Assuming the P-ideal dichotomy, we attempt to isolate those cardinal characteristics of the continuum that are correlated with two well-known consequences of the proper forcing axiom. We find a cardinal invariant $\x$ such that the statement that $\x > {\omega}_{1}$ is equivalent to the statement that $1$, $\omega$, ${\omega}_{1}$, $\omega \times {\omega}_{1}$, and ${\[{\omega}_{1}\]}^{< \omega}$ are the only cofinal types of directed sets of size at most ${\aleph}_{1}$.
We investigate the corresponding problem for the partition relation ${\omega}_{1} \rightarrow ({\omega}_{1}, \alpha)^2$ for all $\alpha < {\omega}_{1}$.
To this effect, we investigate partition relations for pairs of comparable elements of a coherent Suslin tree $\Sa$.
We show that a positive partition relation for such pairs follows from the maximal amount of the proper forcing axiom compatible with the existence of $\Sa$.
As a consequence we conclude that after forcing with the coherent Suslin tree $\Sa$ over a ground model satisfying this relativization of the proper forcing axiom, ${\omega}_{1} ~\rightarrow~{({\omega}_{1}, \alpha)}^{2}$ for all $\alpha < {\omega}_{1}$.
We prove that this positive partition relation for $\Sa$ cannot be improved by showing in $\ZFC$ that $\Sa \not\rightarrow ({\aleph}_{1}, \omega+2)^2$.
\end{abstract}
\maketitle
\section{Introduction} \label{sec:intro}
An interesting, though lesser known, phenomenon in set theory is that cardinal invariants of the continuum can be used to calibrate the strength of various mathematical propositions (that do not unnecessarily involve sets of reals) in the presence of certain kinds of combinatorial dichotomies.
These mathematical statements are invariably consequences of forcing axioms such as $\PFA$ or $\MM$, and they are negated by $\CH$.
The combinatorial dichotomies we are interested in are compatible with $\CH$, but they do keep a considerable amount of the strength of $\PFA$ or $\MM$ by, for example, negating square-principles or reflecting stationary sets. In fact, they have a tendency of pushing several mathematical statements down to concrete questions about combinatorial properties of sets of reals that seem to be expressible in terms of cardinal invariants of the continuum. A typical theorem of the sort we have in mind looks as follows.
\begin{samp}
 Assume $\CD$. Then the following are equivalent:
  \begin{enumerate}
   \item
    $\x > {\omega}_{1}$.
   \item
    $\phi$.
  \end{enumerate}
\end{samp}
Here $\phi$ is some mathematical statement, $\x$ is a cardinal invariant, and $\CD$ is a combinatorial dichotomy that is consistent with $\CH$.
As relations between cardinal invariants have been well-investigated, theorems like this permit us to calibrate the relative strength of various mathematical propositions over the base theory $\ZFC + \CD$.
Examples of $\CD$ include Rado's Conjecture ($\RC$) and the P-Ideal Dichotomy ($\PID$).
The statement $\phi$ can come from different areas of mathematics.
We illustrate this with two recent prototypical examples.
\begin{Theorem}[Todorcevic and Torres-Perez~\cite{victor}] \label{thm:victor}
 Assume $\RC$.
 Then the following are equivalent:
 \begin{enumerate}
  \item
  $\c > {\omega}_{1}$.
  \item
  There are no special ${\omega}_{2}$-Aronszajn trees.
 \end{enumerate}
\end{Theorem}
Here $\c$ is the size of the continuum, the most basic of cardinal invariants.
\begin{Theorem}[Brech and Todorcevic\cite{brechtod}]\label{thm:christina}
 Assume $\PID$.
 Then the following are equivalent:
 \begin{enumerate}
  \item
    $\b > {\omega}_{1}$.
  \item
  Every non-separable Asplund space has an uncountable almost bi-orthogonal system.
 \end{enumerate}
\end{Theorem}
Here $\b$ is the bounding number; its precise definition is given in Section \ref{sec:notation}.
For more regarding the big picture surrounding such results, consult \cite{combdic}.

The purpose of this paper is to add to the analysis of $\PID$ from this point of view; we do not have results about $\RC$ here.
$\PID$ is a well known consequence of $\PFA$ that is consistent with $\CH$ (See Section \ref{sec:notation} below where we give a definition of $\PID$.)
Indeed, it is strong enough to imply many of the consequences of $\PFA$ that don't contradict $\CH$.
For example, $\PID$ implies that there are no Suslin trees (see~\cite{PID}), it implies that $\square(\theta)$ fails for every ordinal $\theta$ of cofinality $> {\omega}_{1}$ (see \cite{GPID}), it implies the Singular Cardinals Hypothesis (see~\cite{viale}), and it implies that ${\square}_{\kappa, \omega}$ fails for all uncountable cardinals $\kappa$ (see~\cite{weaksquares}).
For many consequences of $\PFA$ contradicting $\CH$, $\PID$ tends to reduce the amount of $\PFA$ involved to the hypothesis that some cardinal invariant of the continuum is bigger than ${\omega}_{1}$.
The pseudo-intersection number $\p$ (see Section \ref{sec:notation} for definition), being smaller than most of the usual cardinals, almost always suffices.
We are interested in finding out the precise cardinal invariant which is needed for several specific consequences of $\PFA$. So our general project is twofold:
\begin{gq}
 Given a statement $\phi$ which is a consequence of $\PID + {\MA}_{{\aleph}_{1}}$, find a cardinal invariant $\x$ such that $\phi$ is equivalent to $\x > {\omega}_{1}$ over $\ZFC + \PID$.
\end{gq}
\noindent For example, if $\phi$ is the statement that every non-separable Asplund space has an uncountable almost bi-orthogonal system (see \cite{hmvz} for definitions), then it is a theorem of Todorcevic~\cite{todbio} that $\PID$ + ${\MA}_{{\aleph}_{1}}$ implies $\phi$, while another result of Todorcevic (see \cite{partition}, Chapter 2) shows that $\phi$ implies $\b > {\omega}_{1}$.
So the result of Theorem \ref{thm:christina} above came as an answer to this version of the general problem.

General Problem 1 asks if the influence of $\PFA$ on $\phi$ can be decomposed into a part which is consistent with $\CH$ and into another $\CH$ violating part that is precisely captured by the cardinal invariant $\x$.
This reveals the nature of the combinatorial phenomenon on the reals needed for $\phi$.
Another motivation is that one often has to find a new and sharper proof of $\phi$ in order to accomplish this project.
A slightly less ambitious project is
\begin{gqq} \label{g:gqq}
 Given a statement $\phi$ which is a consequence of $\PID + \p > {\omega}_{1}$, investigate whether $\phi$ is equivalent to $\p > {\omega}_{1}$ over $\ZFC + \PID$.
\end{gqq}
A canonical model for investigating this can be obtained by forcing with a coherent Suslin tree $\Sa$ over a model of $\PFA(\Sa)$ (see Section \ref{sec:notation} for the precise definition of a coherent Suslin tree).
Here $\PFA(\Sa)$ is the maximal amount of $\PFA$ that is consistent with the existence of $\Sa$.
\begin{Def} \label{def:pfas}
Let $\Sa$ be a coherent Suslin tree.
$\PFA(\Sa)$ is the following statement.
If $\P$ is a poset which is proper and preserves $\Sa$ and $\{{D}_{\alpha}: \alpha < {\omega}_{1}\}$ is a collection of dense subsets of $\P$, then there is a filter $G$ on $\P$ such that $\forall \alpha < {\omega}_{1}\[G \cap {D}_{\alpha} \neq 0\]$.
\end{Def}
The consistency of $\PFA(\Sa)$ can be proved assuming the existence of a supercompact cardinal by iterating with countable support all proper posets which preserve $\Sa$.
It is well known that many of the consequences of $\PFA$ hold after forcing with $\Sa$ over a ground model satisfying $\PFA(\Sa)$.
In particular, $\PID$ holds.
Moreover, almost all of the cardinal invariants of the continuum are equal to ${\omega}_{2}$.
However, $\p = {\omega}_{1}$.
Therefore, in this model, one gets most of the consequences of $\PFA$ that are consistent with $\PID + \p = {\omega}_{1}$.
If a consequence of $\PFA$ is \emph{not equivalent} to $\p > {\omega}_{1}$ over $\ZFC + \PID$, then it is very likely to be true in this model.
Thus this model is useful for providing negative answers to the General Problem 2 above.

In this paper, we investigate two well-known consequences of $\PFA$ in view these two general problems.
The first one concerns Tukey theory.
Recall that a poset $\langle D, \leq \rangle$ is \emph{directed} if any two members of $D$ have an upper bound in $D$.
A set $X \subset D$ is \emph{unbounded in $D$} if it doesn't have an upper bound in $D$.
A set $X \subset D$ is said to be \emph{cofinal in $D$} if $\forall y \in D \exists x \in X \[y \leq x \]$.
Given directed sets $D$ and $E$, a map $f: D \rightarrow E$ is called a \emph{Tukey map} if the image (under $f$) of every unbounded subset of $D$ is unbounded in $E$.
A map $g: E \rightarrow D$ is called a \emph{convergent map} if the image (under $g$) of every cofinal subset of $E$ is cofinal in $D$.
It is easy to see that there is a Tukey map $f: D \rightarrow E$ iff there exists a convergent $g: E \rightarrow D$.
When this situation obtains, we say that $D$ is \emph{Tukey reducible} to $E$, and we write $D \; {\leq}_{T} \; E$.
This induces an equivalence relation on directed posets in the usual way: $D \; {\equiv}_{T} \; E$ iff both $D \; {\leq}_{T} \; E$ and $E \; {\leq}_{T} \; D$.
If $D \; {\equiv}_{T} \; E$, we say that $D$ and $E$ are \emph{Tukey equivalent} or have the same \emph{cofinal type}, and this is intended to capture the idea that $D$ and $E$ have ``the same cofinal structure''.
As support for this, it can be shown that $D \; {\equiv}_{T} \; E$ iff there is a directed set $R$ into which both $D$ and $E$ embed as cofinal subsets, so that $D$ and $E$ describe the same {\it cofinal type}, the one of $R$.

These notions first arose in the Moore--Smith theory of convergence studied by general topologists (see \cite{tukeybook} and \cite{isbell}).
The following result of Todorcevic gives a classification of the possible cofinal types of directed posets of size at most ${\aleph}_{1}$ under $\PFA$.
\begin{Theorem}[Todorcevic\cite{cofinal}] \label{thm:five}
Under $\PID + \p > {\omega}_{1}$ there are only 5 Tukey types of size at most ${\aleph}_{1}$: $1$, $\omega$, ${\omega}_{1}$, $\omega \times {\omega}_{1}$, ${\[{\omega}_{1}\]}^{< \omega}$.
\end{Theorem}
Here, the ordering on $\omega \times {\omega}_{1}$ is the product ordering, and ${\[{\omega}_{1}\]}^{< \omega}$ is ordered by inclusion.
In Section \ref{sec:tukey}, we solve General Problem 1 for the statement that $1$, $\omega$, ${\omega}_{1}$, $\omega \times {\omega}_{1}$, and ${\[{\omega}_{1}\]}^{< \omega}$ are the only cofinal types of size at most ${\aleph}_{1}$.
Interestingly, the cardinal invariant that captures this statement turns out not to be one of the commonly occurring ones; rather, it is the minimum of two mutually independent cardinals.
The result in this section answers both Question 24.14 and Question 24.17 of \cite{combdic}, which ask whether the statement that $1$, $\omega$, ${\omega}_{1}$, $\omega \times {\omega}_{1}$, and ${\[{\omega}_{1}\]}^{< \omega}$ are the only cofinal types of size at most ${\aleph}_{1}$ is equivalent over $\ZFC + \PID$ to $\b = {\omega}_{2}$ and to $\p = {\omega}_{2}$ respectively.
The answer to both questions turn out to be ``no''.

In Section \ref{sec:coloring}, we investigate a strong version of the ordinary partition relation on ${\omega}_{1}$, namely the relation ${\omega}_{1} \rightarrow {({\omega}_{1}, \alpha)}^{2}$.
Recall that for an ordinal $\alpha$, ${\omega}_{1} \rightarrow {({\omega}_{1}, \alpha)}^{2}$ means that for any $c: {\[{\omega}_{1}\]}^{2} \rightarrow 2$ \emph{either} there exists $X \in {\[{\omega}_{1}\]}^{{\omega}_{1}}$ such that $c''{\[X\]}^{2} = \{0\}$, \emph{or} there exists $X \subset {\omega}_{1}$ such that $\otp(X) = \alpha$ and $c''{\[X\]}^{2} = \{1\}$.
The Dushnik-Miller theorem says that ${\omega}_{1}~\rightarrow~{({\omega}_{1}, \omega)}^{2}$ (see \cite{DusMiller}), and its strengthening proved by Erd{\H{o}}s and Rado states that ${\omega}_{1}~\rightarrow~{({\omega}_{1}, \omega + 1)}^{2}$ (see \cite{ErdRado}), while the classical coloring of Sierpinski \cite{Sierp} shows that ${\omega}_{1}~\rightarrow~{({\omega}_{1}, {\omega}_{1})}^{2}$ is false.
Thus the following theorem of Todorcevic~\cite{positive} gives the strongest possible version of the ordinary partition relation on ${\omega}_{1}$.
\begin{Theorem}[Todorcevic] \label{omegaone}
$\PID + \p > {\omega}_{1}$ implies that ${\omega}_{1} \rightarrow {({\omega}_{1}, \alpha)}^{2}$, for every $\alpha < {\omega}_{1}$.
\end{Theorem}
\noindent This should be compared with the following result of Todorcevic~\cite{partition} (Chapter 2) that is relevant to General Problem 1.
\begin{Theorem}[Todorcevic]\label{beomegaone}
 $\b = {\omega}_{1}$ implies  ${\omega}_{1} ~\not \rightarrow~{({\omega}_{1}, \omega + 2)}^{2}$
\end{Theorem}
In Section \ref{sec:coloring} we solve General Problem 2 for the statement that ${\omega}_{1} \rightarrow {({\omega}_{1}, \alpha)}^{2}$, for every $\alpha < {\omega}_{1}$ by showing that it holds after forcing with the coherent Suslin tree $\Sa$ over a ground model satisfying $\PFA(\Sa)$.
This shows that $\p > {\omega}_{1}$ is not equivalent over $\ZFC + \PID$ to the statement that ${\omega}_{1} \rightarrow {({\omega}_{1}, \alpha)}^{2}$ for every $\alpha < {\omega}_{1}$ so it remains to look for another cardinal invariant of the continuum that would capture this partition relation for ${\omega}_{1}$.
\section{Notation} \label{sec:notation}
We setup some basic notation that will be used throughout the paper.
``$a \subset b$'' means $\forall x\[x \in a \implies x \in b\]$, so the symbol ``$\subset$'' does not denote proper subset.
``$\forallbutfin$'' means for all but finitely many and ``$\existsinf$'' stands for there exists infinitely many.

The following well known cardinal invariants will occur throughout the paper.
For functions $f, g \in \BS$, $f \; {<}^{\ast} \; g$ means $\forallbutfin n \in \omega \[f(n) < g(n)\]$.
A set $F \subset \BS$ is said to be \emph{unbounded} if there is no $g \in \BS$ such that $\forall f \in F \[f \; {<}^{\ast} \; g\]$.
For sets $a$ and $b$, $a \; {\subset}^{\ast} \; b$ iff $a \setminus b$ is finite.
A family $F \subset \cube$ is said to have the \emph{finite intersection property (FIP)} if for any $A \in {\[ F \]}^{< \omega}$, ${\bigcap}{A}$ is infinite.
\begin{Def} \label{def:bandp}
$\c$ denotes ${2}^{\omega}$. Additionally,
\begin{align*}
& \p = \min\left\{\lc F \rc: F \subset \cube \wedge F \ \text{has the FIP} \ \wedge \neg \exists b \in \cube \; \forall a \in F\[b \; {\subset}^{\ast} \; a\]\right\}. \\
& \b = \min\left\{\lc F \rc: F \subset \BS \wedge F \ \text{is unbounded}\right\}.
\end{align*}
$\cov(\M)$ is the least $\kappa$ such that $\R$ can be covered by $\kappa$ many meager sets.
\end{Def}
It is easy to show ${\omega}_{1} \leq \p \leq \b \leq \c$.
Moreover, $\p \leq \cov(\M) \leq \c$, while $\b$ and $\cov(\M)$ are independent.

We will frequently make use of elementary submodels.
We will simply write ``$M \prec H{(\theta)}$'' to mean ``$M$ is an elementary submodel of $H{(\theta)}$, where $\theta$ is a regular cardinal that is large enough for the argument at hand''.

Recall that a Suslin tree is an ${\omega}_{1}$ tree with no uncountable chains or antichains.
Throughout the paper, we work with a fixed Suslin tree $\Sa$
We assume that $\Sa$ is a \emph{coherent strongly homogeneous Suslin tree}.
More precisely, this means that
\begin{enumerate}
 \item
   $\Sa$ is a Suslin tree and is a subtree of ${\omega}^{< {\omega}_{1}}$.
 \item
   for each $s \in \Sa$, $\existsinf n \in \omega \[{s}^{\frown}{\langle n \rangle} \in \Sa\]$ and $\{t \in \Sa: t \geq s\}$ is uncountable.
 \item
   $\forall s, t, \in \Sa \[\lc \xi \in \dom(s) \cap \dom(t): s(\xi) \neq t(\xi) \rc < \omega \]$ (Coherence).
  \item
   For each $\xi < {\omega}_{1}$ and $s, t \in {\Sa}_{\xi}$, there is an automorphism $\phi: \Sa \rightarrow \Sa$ such that $\phi(s) = t$ and
   $\forall \alpha \geq \xi \forall u \in {\Sa}_{\alpha}\[\phi(u) = \phi(u \restrict \xi) \cup u \restrict [\xi, \alpha)\]$ (Strong homogeneity).
\end{enumerate}
Thus fix once and for all a coherent strongly homogeneous Suslin tree $\Sa$.

We will use the following notation when dealing with $\Sa$.
For $t \in \Sa$, $\hgt(t)$ is the height of $t$.
For $t \in \Sa$, $\pred(t)$ denotes \emph{the set of predecessors of $t$}, that is $\{s \in \Sa: s \leq t\}$.
For a set $X$ and $t \in \Sa$, ${\pred}_{X}(t) = \pred(t) \cap X$.
Dually, $\cone(t)$ denotes \emph{the cone above $t$}, for all $t \in \Sa$.
In other words, $\cone(t) = \{u \in \Sa: t \leq u\}$.
Similarly, for a set $X$ and $t \in \Sa$, ${\cone}_{X}(t) = \cone(t) \cap X$.
Next, for $t \in \Sa$, $\succc(t) = \{u \in \Sa: u > t \ \text{and} \ \hgt(u) = \hgt(t) + 1\}$.
Once again, for $t \in \Sa$ and a set $X$, ${\succc}_{X}(t) = \succc(t) \cap X$.
For any non-empty $F \subset \Sa$, $\bigwedge F$ denotes the greatest lower bound  in $\Sa$ of the elements of $F$.

We will be studying colorings of the pairs of comparable elements of $\Sa$.
We setup some basic notation relevant to such colorings here.
For any $A, B \subset \Sa$, $A \otimes B = \{\{a, b\}: a \in A \ \text{and} \ b \in B \ \text{and} \ a < b\}$.
${A}^{\[2\]} = A \otimes A$.
The following variation of ${A}^{\[2\]}$ will also be important in Section \ref{sec:coloring}.
Let $Y \subset \Sa$ and $g: Y \rightarrow \Sa$.
Then ${Y}^{\[2\]}_{g}$ denotes $\{\{a, b\}: a, b \in Y \ \text{and} \ a < b \ \text{and} \ g(a) \leq b\}$.
If $S \subset \Sa$ and $c: {S}^{\[2\]} \rightarrow 2$ is a coloring, then ${K}_{i,c} = \left\{\{s, t\} \in {S}^{\[2\]}: c(\{s, t\}) = i \right\}$, for each $i \in 2$.
We will often omit the subscript ``$c$'' when it is clear from the context.

We will use a $C$-sequence in the proof of Theorem \ref{thm:main}.
Recall that $\langle {c}_{\alpha}: \alpha < {\omega}_{1}\rangle$ is called a \emph{$C$-sequence} if for each $\alpha < {\omega}_{1}$,
\begin{enumerate}
 \item
 ${c}_{\alpha} \subset \alpha$.
 \item
 If $\alpha$ is a limit ordinal, then $\otp({c}_{\alpha}) = \omega$ and $\sup({c}_{\alpha}) = \alpha$.
 \item
 If $\beta = \alpha + 1$, then ${c}_{\beta} = \{\alpha\}$.
\end{enumerate}
Given a $C$-sequence $\langle {c}_{\alpha}: \alpha < {\omega}_{1}\rangle$, it is sometimes useful to think of each ${c}_{\alpha}$ as a function.
For a fixed $0 < \beta < {\omega}_{1}$ and $n \in \omega$, we adopt the following conventions.
If $\beta = \alpha + 1$, then ${c}_{\beta}(n)$ is the unique element of ${c}_{\beta}$, namely $\alpha$.
If $\beta$ is a limit ordinal, then ${c}_{\beta}(n)$ is the $n$th element of ${c}_{\beta}$.

In the proof of Theorem \ref{thm:coloring}, we will use the following notation for subtrees of ${\omega}^{< \omega}$.
Given $T \subset {\omega}^{< \omega}$ which is a subtree, $\[T\] = \{f \in \BS: \forall n \in \omega \[f \restrict n \in T\]\}$.
For $\sigma \in T$, ${\succc}_{T}(\sigma)$ denotes $\{n \in \omega: {\sigma}^{\frown}{\langle n \rangle} \in T\}$.
Note that this departs from the definition of $\succc(s)$ when $s$ is a member of $\Sa$.
However, since it will be clear when we are talking about members of $\Sa$ and when we are referring to elements of some subtree $T$ of ${\omega}^{< \omega}$, we hope that this will not cause any confusion.
\section{Partition relations after forcing with a coherent Suslin tree} \label{sec:coloring}
In this section we investigate partition relations for the pairs of comparable elements of $\Sa$.
Partition relations for the pairs of comparable elements of a Suslin tree were studied by M{\'a}t{\'e}~\cite{matesuslin}, and for more general partial orders by Todorcevic~\cite{posetpartition}.
We prove a positive partition relation for ${\Sa}^{\[2\]}$ assuming $\PFA(\Sa)$ (Theorem \ref{thm:main}).
This result in an analogue of Todorcevic's theorem from \cite{positive} that $\PFA$ implies that ${\omega}_{1} \rightarrow {({\omega}_{1}, \alpha)}^{2}$, for every $\alpha < {\omega}_{1}$.
However, it is not a perfect analogue.
This is explained by Theorem \ref{thm:coloring}, which establishes a negative partition relation for ${\Sa}^{\[2\]}$ in $\ZFC$.
This negative partition relation may be seen as a $\ZFC$ analogue of the relation ${\omega}_{1} \not\rightarrow({\omega}_{1}, \omega+2)^2$ for the pairs of comparable elements of $\Sa$.
It is somewhat surprising that such a result can be proved not going beyond $\ZFC$.
A corollary of the positive partition relation proved in this section is that ${\omega}_{1} \rightarrow {({\omega}_{1}, \alpha)}^{2}$ for every $\alpha < {\omega}_{1}$ after forcing with $\Sa$ over a model of $\PFA(\Sa)$.

In this section, all the trees we deal with will be subsets (though not subtrees) of $\Sa$.
Of course, given $T \in {\[\Sa\]}^{{\omega}_{1}}$, $\langle T, \leq \rangle$ is an ${\omega}_{1}$ tree with no uncountable chains or antichains.
However to avoid some trivialities, we make the following definition
\begin{Def} \label{def:suslin}
 $T \subset \Sa$ is called a \emph{Suslin tree} if $T$ is uncountable and
 \begin{enumerate}
  \item
    $\exists \min(T) \in T \forall x \in T \[\min(T) \leq x\]$.
  \item
    $\forall x \in T \[\{y \in T: y \geq x \} \ \text{is uncountable}\]$.
 \end{enumerate}
Note that we are \emph{not} requiring $T$ to be a normal tree.
In general, $T$ will not be a subtree of $\Sa$.
 \end{Def}
 Obviously, $\forall Y \in {\[\Sa\]}^{{\omega}_{1}} \exists T \in {\[Y\]}^{{\omega}_{1}}\[T \ \text{is a Suslin tree}\]$.
 We will also use the following consequence of a well-known lemma of Todorcevic (see \cite{posetpartition}).
 \begin{Lemma} \label{lem:regressive}
 Let $R \subset \Sa$ be a Suslin tree.
 Suppose $f: R \setminus \{\min(R)\} \rightarrow R$ is a function such that $\forall x \in R \setminus \{\min(R)\}\[f(x) < x\]$.
 Then $\exists U \in {\[R \setminus \{\min(R)\}\]}^{{\omega}_{1}} \exists s \in R \forall x \in U \[f(x) = s\]$.
 \end{Lemma}
 Another useful fact about Suslin trees that is easy to verify is the following.
 \begin{Lemma} \label{lem:dense}
  Let $T \subset \Sa$ be a Suslin tree.
  If $X \in {\[T\]}^{{\omega}_{1}}$, then there exists $x \in X$ such that $X$ is dense above $x$ in $T$.
 \end{Lemma}
 We come to the main result that will be established in this section. Our claim that ${\omega}_{1} \rightarrow {({\omega}_{1}, \alpha)}^{2}$ for all $\alpha < {\omega}_{1}$ after forcing with $\Sa$ will follow from this result.
 Theorem \ref{thm:main} gives a positive partition relation for ${\Sa}^{\[2\]}$ under $\PFA(\Sa)$.
 \begin{Theorem} \label{thm:main}
  Assume $\PFA(\Sa)$.
  Let $S \in {\[\Sa\]}^{{\omega}_{1}}$ and $c: {S}^{\[2\]} \rightarrow 2$.
  Then either there exist $Y \in {\[S\]}^{{\omega}_{1}}$ and $g: Y \rightarrow \Sa$ such that $\forall y \in Y \[g(y) \geq y\]$ and ${Y}^{\[2\]}_{g} \subset {K}_{0}$ or for each $\alpha < {\omega}_{1}$, there exists $s \in S$ and $B \subset {\pred}_{S}(s)$ such that $\otp(B) = \alpha$ and ${B}^{\[2\]} \subset {K}_{1}$.
 \end{Theorem}
 \begin{Cor} \label{cor:main}
  $\PFA(\Sa)$ implies that the coherent Suslin tree $\Sa$ forces ${\omega}_{1} ~\rightarrow~{({\omega}_{1}, \alpha)}^{2}$ to hold.
 \end{Cor}
 \begin{proof}[Proof(assuming Theorem \ref{thm:main}).]
  Let $\mathring{f} \in {\V}^{\Sa}$ be such that $\forces \mathring{f}: {\[{\omega}_{1}\]}^{2} \rightarrow 2$.
  Fix $s \in \Sa$.
  Suppose that $s \forces \neg\exists A \in {\[{\omega}_{1}\]}^{{\omega}_{1}}\[\mathring{f}''{\[A\]}^{2} = \{0\}\]$.
  We will show that $s \forces \forall \alpha < {\omega}_{1} \exists A \subset {\omega}_{1}\[\otp(A) = \alpha \wedge \mathring{f}''{\[A\]}^{2} = \{1\}\]$.
  Fix $\alpha < {\omega}_{1}$ and $t \geq s$.
  Let $\theta$ be a sufficiently large regular cardinal and let $\langle {M}_{\xi}: \xi < {\omega}_{1}\rangle$ be an increasing continuous $\in$-chain of countable elementary submodels of $H(\theta)$, with ${M}_{0}$ containing all the relevant objects.
  For each $\xi < {\omega}_{1}$, put ${\delta}_{\xi} = {M}_{\xi} \cap {\omega}_{1}$.
  It is clear that for any $\xi < {\omega}_{1}$, $\beta < \gamma < {\delta}_{\xi}$, and $x \in \Sa$, if $\hgt(x) \geq {\delta}_{\xi}$, then there exists $i \in 2$ such that $x \forces \mathring{f}(\{\beta, \gamma\}) = i$.
  Put $S = \{x \in {\cone}_{\Sa}(t): \exists \xi < {\omega}_{1}\[\hgt(x) = {\delta}_{\xi + 1}\]\}$.
  For $x \in S$, let ${\xi}_{x} < {\omega}_{1}$ be such that $\hgt(x) = {\delta}_{{\xi}_{x} + 1}$.
  Define $c: {S}^{[2]} \rightarrow 2$ as follows.
  Given $x, y \in S$ with $x < y$, let $c(\{x, y\}) \in 2$ be such that $y \forces \mathring{f}(\{{\delta}_{{\xi}_{x}}, {\delta}_{{\xi}_{y}}\}) = c(\{x, y\})$.
  First suppose that there are $Y \in {\[S\]}^{{\omega}_{1}}$ and $g: Y \rightarrow \Sa$ with $\forall y \in Y \[g(y) \geq y\]$ and ${Y}^{\[2\]}_{g} \subset {K}_{0}$.
  Choose $x \in Y$ such that $Y$ is dense above $x$ in $\Sa$.
  If $G$ is a $(\V, \Sa)$-generic filter with $x \in  G$, then in $\V\[G\]$, it is possible to find an uncountable $Z \subset G \cap Y$ such that $\forall y, z \in Z\[y < z \implies g(y) \leq z \]$.
  Now, $\{{\delta}_{{\xi}_{y}}: y \in Z\}$ is an uncountable 0-homogeneous set for $\mathring{f}\[G\]$.
  As $s \in G$, this contradicts the hypothesis on $s$.
  So by Theorem \ref{thm:main}, there is $x \in S$ and $B \subset {\pred}_{S}(x)$ such that $\otp(B) = \alpha$ and ${B}^{\[2\]} \subset {K}_{1}$.
  Then letting $A = \{{\delta}_{{\xi}_{b}}: b \in B\}$, $\otp(A) = \alpha$ and $x \forces \mathring{f}''{\[A\]}^{2} = \{1\}$.
  As $x\geq t$, this completes the proof.
 \end{proof}
 \begin{Def} \label{def:poset}
 Let $\chi$ be a sufficiently large regular cardinal.
 Let $S \in {\[\Sa\]}^{{\omega}_{1}}$ and $c: {S}^{\[2\]} \rightarrow 2$.
 Define a poset $\P(S, c)$ as follows.
 A condition in $\P(S, c)$ is a pair $p = \langle {F}_{p}, {\N}_{p} \rangle$ such that
  \begin{enumerate}
   \item
    ${F}_{p} \in {\[S\]}^{< \omega}$ such that ${{F}_{p}}^{\[2\]} \subset {K}_{0}$.
   \item
    ${\N}_{p}$ is a finite $\in$-chain of countable elementary submodels of $H(\chi)$ that contain all the relevant objects.
   \item
    $\forall s, t \in {F}_{p}\exists M \in {\N}_{p}\[\lc M  \cap \{s, t\} \rc = 1\]$.
    \item
      $\exists M \in {\N}_{p}\[M \cap {F}_{p} = 0\]$.
  \end{enumerate}
  For $p, q \in \P(S, c)$, $q \leq p$ iff ${F}_{q} \supset {F}_{p}$ and ${\N}_{q} \supset {\N}_{p}$.
 \end{Def}
 \begin{Lemma} \label{lem:density}
 Let $S \in {\[\Sa\]}^{{\omega}_{1}}$ and $c: {S}^{\[2\]} \rightarrow 2$.
 For each $\alpha < {\omega}_{1}$, put ${D}_{\alpha} = \{q \in \P(S, c): \exists t \in {F}_{q} \[\hgt(t) > \alpha\]\}$.
 ${D}_{\alpha}$ is a dense subset of $\P(S, c)$.
 \end{Lemma}
 \begin{proof}
  Fix $\alpha < {\omega}_{1}$ and $p \in \P(S, c)$.
  Let $\{{M}_{0}, \dotsc, {M}_{l}\}$ enumerate ${\N}_{p}$ in $\in$-increasing order.
  By (4) of Definition \ref{def:poset}, ${M}_{0} \cap {F}_{p} = 0$.
  Put $\delta = \max\{\alpha, {M}_{l} \cap {\omega}_{1}\}$.
  We will find $v \in S$ such that $\hgt(v) > \delta$ and $\neg\exists s \in {F}_{p}\[s \leq v\]$.
  Suppose that this is not possible.
  Fix $T \in {M}_{0} \cap {\[S\]}^{{\omega}_{1}}$ such that $\forall u \in T \[\{v \in T: v \geq u\} \ \text{is uncountable}\]$.
  Suppose $u \in T \cap {M}_{0}$.
  There is a $v \in T$ so that $\hgt(v) > \delta$ and $v \geq u$.
  By the assumption there exists $s \in {F}_{p}$ such that $s \leq v$.
  It follows that $u \leq s$.
  Thus $\forall u \in T \cap {M}_{0}\exists s \in {F}_{p}\[u \leq s\]$.
  It follows that $T$ is an uncountable subset of $\Sa$ with the property that $\forall A \in {\[T\]}^{\leq \omega} \exists F \in {\[S\]}^{< \omega}\forall u \in A \exists s \in F\[u \leq s\]$.
  However it is impossible to have such an uncountable subset of a Suslin tree.

  Now, fix $v \in S$ such that $\hgt(v) > \delta$ and $\neg\exists s \in {F}_{p}\[s \leq v\]$.
  Put ${F}_{q} = {F}_{p} \cup \{v\}$.
  Find a countable ${M}_{l + 1} \prec H(\chi)$ with $v, p \in {M}_{l + 1}$.
  Put ${\N}_{q} = {\N}_{p} \cup \{{M}_{l + 1}\}$.
  Then $q \leq p$ and $q \in {D}_{\alpha}$.
 \end{proof}
 \begin{Lemma} \label{lem:submain}
  Let $\chi$ be the cardinal fixed in Definition \ref{def:poset}.
  Fix $M \prec H(\chi)$ countable with $\Sa \in M$.
  Let $D \subset \Sa$ with $D \in M$.
  Let $L \subset \Sa$ be such that $\{\hgt(x): x \in L \cap M \}$ is unbounded in $M \cap {\omega}_{1}$.
  Suppose that $\forall x \in L \cap M \exists y \in D \[x \leq y\]$.
  Then there exists $x \in L \cap M$ such that $D$ is dense above $x$ in $\Sa$.
  Moreover, if there exists $s \in \Sa$ such that $L = {\pred}_{\Sa}(s)$, then $\{\hgt(x): x \in L \cap M \cap D\}$ is unbounded in $M \cap {\omega}_{1}$.
 \end{Lemma}
 \begin{proof}
  Put $\delta = M \cap {\omega}_{1}$.
  Put $E = \{x \in \Sa: {\cone}_{D}(x) = 0\}$.
  $E \in M$.
  So there exists $A \in M$ such that $A \subset E$, $A$ is an antichain, and $A$ is maximal with respect to these two properties.
  As $A$ is countable, find $\alpha < \delta$ such that $A \subset {\Sa}_{< \alpha}$.
  Let $x \in L \cap M$ be such that $\hgt(x) \geq \alpha$.
  If $D$ is not dense above $x$ in $\Sa$, then there is $s \in \Sa$ such that $s \geq x$ and ${\cone}_{D}(s) = 0$.
  Thus $s \in E$ and is comparable to some $a \in A$.
  It follows that $a \leq x$.
  However, by hypothesis, there is $y \in D$ with $x \leq y$.
  $y \in {\cone}_{D}(a)$, contradicting $a \in E$.

  For the second statement assume that $L = {\pred}_{\Sa}(s)$ for some $s \in \Sa$, and fix $\alpha < M \cap {\omega}_{1}$.
  By the first statement, fix $x \in L \cap M$ such that $D$ is dense above $x$ in $\Sa$.
  Note that ${\cone}_{D}(x) \in M$ and that it is an uncountable set.
  Put $B = \{y \in {\cone}_{D}(x): \hgt(y) > \alpha \} \in M$.
  Choose $A \in M$ such that $A \subset B$, $A$ is an antichain, and $A$ is maximal with respect to these two properties.
  As $A$ is countable, fix $\beta < M \cap {\omega}_{1}$ such that $A \subset {\Sa}_{< \beta}$.
  Fix $t \in L \cap M$ with $\hgt(t) > \max\{\alpha, \beta, \hgt(x)\}$.
  Thus $t \geq x$ and there is $y \in D$ with $y \geq t$.
  Since $y \in B$, there is $a \in A$ such that $a \leq y$.
  It follows that $a \leq t \leq s$.
  Therefore, $a \in L \cap D \cap M$ and $\hgt(a) > \alpha$.
 \end{proof}
 \begin{Lemma} \label{lem:main1}
  Let $T \subset \Sa$ be a Suslin tree.
  Suppose $c: {T}^{\[2\]} \rightarrow 2$.
  Either there exists $X \in {\[T\]}^{{\omega}_{1}}$ such that $\P\left(X, c \restrict {X}^{\[2\]}\right)$ is proper and preserves $\Sa$ or for each $X \in {\[T\]}^{{\omega}_{1}}$ there exists ${x}_{0} \in X$, $Y \in {\[X\]}^{{\omega}_{1}}$, a sequence $\langle {F}_{\alpha}: \alpha < {\omega}_{1}\rangle$, and a function $g: Y \rightarrow \Sa$ such that
  \begin{enumerate}
   \item
    For each $\alpha < {\omega}_{1}$, ${F}_{\alpha}$ is a non-empty finite subset of $X$ such that $\hgt\left(\bigwedge {F}_{\alpha}\right) > \alpha$ (keep in mind that $\bigwedge {F}_{\alpha}$ may not be in $T$).
   \item
    $Y$ and $\{\bigwedge {F}_{\alpha}: \alpha < {\omega}_{1}\}$ are both dense above ${x}_{0}$ in $\Sa$.
   \item
    $\forall x \in Y\[g(x) \geq x\]$ and for each $\alpha < {\omega}_{1}$ and $s \in {\pred}_{Y}\left(\bigwedge {F}_{\alpha}\right) \cap {\Sa}_{< \alpha}$, if $g(s) \leq \bigwedge {F}_{\alpha}$, then $\exists t \in {F}_{\alpha}\[c(\{s, t\}) = 1\]$.
  \end{enumerate}
 \end{Lemma}
\begin{proof}
 Fix a sufficiently large regular $\theta > \chi$.
 Suppose $X \in {\[T\]}^{{\omega}_{1}}$.
 For ease of notation, write ${\P}_{X}$ for $\P(X, c \restrict {X}^{\[2\]})$.
 If for any countable $M \prec H(\theta)$ containing all the relevant objects and any ${p}_{0} \in M \cap {\P}_{X}$, there exists $p \leq {p}_{0}$ such that $\forall {t}_{0} \in {\Sa}_{M \cap {\omega}_{1}}\[\langle p, {t}_{0}\rangle \ \text{is} \ \left( M, {\P}_{X} \times \Sa \right) \ \text{generic}\]$, then ${\P}_{X}$ is proper and preserves $\Sa$.
 Assume that this fails and fix $M \prec H(\theta)$ and ${p}_{0} \in M  \cap {\P}_{X}$ witnessing this.
 Put $\delta = M \cap {\omega}_{1}$.
 Put ${F}_{p} = {F}_{{p}_{0}}$ and ${\N}_{p} = {\N}_{{p}_{0}} \cup \{M \cap H(\chi)\}$.
 Then $p = \langle {F}_{p}, {\N}_{p} \rangle \in {\P}_{X}$ and extends ${p}_{0}$.
 Let ${t}_{0} \in {\Sa}_{\delta}$ and let $D \in M$ be a dense open subset of ${\P}_{X} \times \Sa$ such that $D \cap M$ is not predense below $\langle p, {t}_{0}\rangle$.
 Fix $\langle q, t \rangle \leq \langle p, {t}_{0}\rangle$ which is incompatible with every element of $D \cap M$.
 By extending it if necessary we may assume that $\langle q, t \rangle \in D$, that $\forall N \in {\N}_{q}\[\hgt(t) > N \cap {\omega}_{1}\]$, and that $\exists N \in M \cap {\N}_{q}\[M \cap {F}_{q} \subset N\]$.
 Put ${F}_{{q}_{0}} = {F}_{q} \cap M$ and ${\N}_{{q}_{0}} = {\N}_{q} \cap M$.
 It is clear that ${q}_{0} = \langle {F}_{{q}_{0}}, {\N}_{{q}_{0}} \rangle \in {\P}_{X}$ and that ${q}_{0} \in M$.

 Let $\{{N}^{\ast}_{0}, \dotsc, {N}^{\ast}_{k} \}$ enumerate ${\N}_{q} \setminus {\N}_{{q}_{0}}$ in $\in$-increasing order.
 Let $F = {F}_{q} \setminus {F}_{{q}_{0}}$.
 If $F = 0$, then consider ${D}^{\ast}$, collection of all ${t}^{\ast} \in \Sa$ for which there exists $\{{N}_{0}, \dotsc, {N}_{k}\}$ such that
 \begin{enumerate}
  \item[(4)]
    $\{{N}_{0}, \dotsc, {N}_{k}\}$ is an $\in$-chain of countable elementary submodels of $H(\chi)$ containing the relevant objects, and containing ${q}_{0}$.
  \item[(5)]
   $\langle \langle {F}_{{q}_{0}}, {\N}_{{q}_{0}} \cup \{{N}_{0}, \dotsc, {N}_{k}\}\rangle, {t}^{\ast} \rangle \in D$.
 \end{enumerate}
 ${D}^{\ast} \in M$ and it follows from Lemma \ref{lem:submain} (applied to $M \cap H(\chi)$) that there exists ${t}^{\ast} \in {D}^{\ast} \cap M \cap {\pred}_{\Sa}(t)$.
 If $\{{N}_{0}, \dotsc, {N}_{k}\} \in M$ witnesses (4) and (5) for ${t}^{\ast}$, then $\langle \langle {F}_{{q}_{0}}, {\N}_{{q}_{0}} \cup \{{N}_{0}, \dotsc, {N}_{k}\} \rangle, {t}^{\ast} \rangle \in D \cap M$ and is compatible with $\langle q, t\rangle$.
 As this contradicts the choice of $\langle q, t \rangle$, we may assume that $F \neq 0$.
 Let $\{{z}_{0}, \dots, {z}_{m}\}$ enumerate $F$ in increasing order of their heights.
 For each $0 \leq i, j \leq m$, put ${G}_{i, j} = \{\xi \in \dom({z}_{i}) \cap \dom({z}_{j}): {z}_{i}(\xi) \neq {z}_{j}(\xi)\}$ and ${G}_{i} = \{\xi \in \dom({z}_{i}) \cap \dom(t): t(\xi) \neq {z}_{i}(\xi)\}$.
 These sets are all finite.
 Choose $\zeta < \delta$ such that for each $0 \leq i, j \leq m$, ${G}_{i, j} \cap M \subset \zeta$, ${G}_{i} \cap M \subset \zeta$, and for each $N \in {\N}_{{q}_{0}}\[N \cap {\omega}_{1} \in \zeta\]$.
 For $0 \leq i \leq m$, put ${y}_{i} = {z}_{i} \restrict \zeta$, and put ${t}^{\ast} = t \restrict \zeta$.
 For each $0 \leq i \leq m$, choose an automorphism ${\phi}_{i}: \Sa \rightarrow \Sa$ such that ${\phi}_{i}({t}^{\ast}) = {y}_{i}$ and for all $\alpha \geq \zeta$ and all $u \in {\Sa}_{\alpha}\[{\phi}_{i}(u) = {\phi}_{i}(u \restrict \zeta) \cup u \restrict [\zeta, \alpha) \]$.
 We may assume that ${\phi}_{i} \in M$.
 It is easy to see that for any $s \in M$ with ${t}^{\ast} \leq s \leq t$ and any $0 \leq i \leq m$, ${y}_{i} \leq {\phi}_{i}(s) \leq {z}_{i}$.
 Also, for any $\zeta \leq \alpha < \delta$, $0 \leq i, j \leq m$, and $s \in {\Sa}_{\alpha} \cap {\cone}_{\Sa}({y}_{i})$, if $s \leq {z}_{j}$, then $s = {\phi}_{i}(t \restrict \alpha)$.
 There are two types of points in $F$ that we must deal with.
 Put ${I}_{0} = \{0 \leq i \leq m: {G}_{i} \setminus M = 0\}$ and ${I}_{1} = \{0 \leq i \leq m: {G}_{i} \setminus M \neq 0\}$.
 Observe that if $i \in {I}_{1}$, then ${\phi}_{i}(t \restrict \hgt({z}_{i})) \neq {z}_{i}$.
 On the other hand if $i \in {I}_{0}$, then ${\phi}_{i}(t \restrict \hgt({z}_{i})) = {z}_{i}$.
 For any $s \geq {t}^{\ast}$, define a two-player game $\G(s)$ as follows.
 The game lasts $m + 2$ moves.
 In the first move I chooses ${s}_{0} \geq s$ and II responds with a pair $\langle {x}_{0}, {u}_{0}\rangle$ that satisfies ${s}_{0} \leq {x}_{0} \leq {u}_{0}$.
 In the next move I chooses ${s}_{1}$ with ${u}_{0} \leq {s}_{1}$.
 At the end of $m + 2$ moves the players have constructed a sequence
 \begin{align*}
  {s}_{0}, \langle {x}_{0}, {u}_{0} \rangle, \dotsc, {s}_{m + 1}, \langle {x}_{m + 1}, {u}_{m + 1}\rangle
 \end{align*}
 such that $s \leq {s}_{0} \leq {x}_{0} \leq {u}_{0} \leq {s}_{1} \leq {x}_{1} \leq {u}_{1} \leq \dotsb \leq {s}_{m + 1} \leq {x}_{m + 1} \leq {u}_{m + 1}$.
 We say that II wins $\G(s)$ if there exist $\{{N}_{0}, \dotsc, {N}_{k}\}$ and $\{{v}_{i}: i \in {I}_{1}\}$ such that
 \begin{enumerate}
  \item[(6)]
    $\{{N}_{0}, \dotsc, {N}_{k}\}$ is an $\in$-chain of countable elementary submodels of $H(\chi)$ containing the relevant objects, and containing ${q}_{0}$.
  \item[(7)]
    $\forall i \in {I}_{1}\[{v}_{i} \in {\Sa}_{\hgt({x}_{i})} \cap {\cone}_{\Sa}({y}_{i}) \ \text{and} \ {v}_{i} \neq {\phi}_{i}({x}_{i})\]$.
  \item[(8)]
    $\langle \langle {F}_{{q}_{0}} \cup \{{v}_{i}: i \in {I}_{1}\} \cup \{{\phi}_{i}({x}_{i}): i \in {I}_{0}\}, {\N}_{{q}_{0}} \cup \{{N}_{0}, \dotsc, {N}_{k} \}\rangle, {x}_{m + 1}\rangle \in D$.
 \end{enumerate}
 Let ${D}^{\ast} = \{s \geq {t}^{\ast}: \text{II has a winning strategy in} \ \G(s)\}$. Then ${D}^{\ast} \in M$. Now we
 \begin{Claim} \label{claim:1}
  $\forall {s}^{\ast} \in M \cap {\pred}_{\Sa}(t) \exists s \in {D}^{\ast}\[{s}^{\ast} \leq s\]$.
 \end{Claim}
 \begin{proof}
 Suppose not.
 Fix ${s}^{\ast} \in M \cap {\pred}_{\Sa}(t)$ with ${s}^{\ast} \geq {t}^{\ast}$ such that for each $s \geq {s}^{\ast}$, I has a winning strategy in $\G(s)$.
 Fix $\Sigma \in M$ such that for each $s \geq {s}^{\ast}\[\Sigma(s) \ \text{is a winning strategy for I in} \ \G(s)\]$.
 Consider ${D}_{0} =$
 \begin{align*}
  \{{s}_{0} \in \Sa: \exists s \geq {s}^{\ast}\[{s}_{0} \ \text{is the first move of I according to} \ \Sigma(s)\]\}.
 \end{align*}
 ${D}_{0} \in M$ and applying Lemma \ref{lem:submain}, fix ${s}_{0} \in M \cap {\pred}_{\Sa}(t)$ and ${s}^{\ast} \leq s \leq {s}_{0}$ such that ${s}_{0}$ is the first move of I according to $\Sigma(s)$.
 Observe that $\Sigma(s) \in M$ and we can think of $\Sigma(s)$ as a subset of ${\Sa}^{< \omega} \times \Sa$.
 Hence $\Sigma(s) \in M \cap H(\chi)$ and hence $\Sigma(s) \in {N}^{\ast}_{i}$ for all $0 \leq i \leq k$.
 Now define a run of $\G(s)$ according to $\Sigma(s)$ as follows.
 Fix $0 \leq i \leq m$ and suppose that for all $j < i$, ${s}_{j}$ and $\langle {x}_{j}, {u}_{j} \rangle$ have already been specified in such a way that ${s}_{j}$ is according to $\Sigma(s)$, ${x}_{j} = t \restrict \hgt({z}_{j})$, and if ${s}_{i}$ is the continuation of this play according to $\Sigma(s)$, then ${s}_{i} \in {\pred}_{\Sa}(t)$ and $\hgt({s}_{i}) < \hgt({z}_{i})$ (when $i = 0$ this is satisfied because ${s}_{0} \in M \cap {\pred}_{\Sa}(t)$, and since ${z}_{0} \notin M$, $\hgt({z}_{0}) \geq \delta > \hgt({s}_{0})$).
 Let $0 \leq l \leq k$ be minimal such that ${z}_{i} \in {N}^{\ast}_{l}$.
 Note that $\langle {s}_{0}, \langle {x}_{0}, {u}_{0} \rangle, \dotsc, {s}_{i}\rangle \in {N}^{\ast}_{l}$, that $t \restrict \hgt({z}_{i}) \in {N}^{\ast}_{l}$, and that ${s}_{i} \leq t \restrict \hgt({z}_{i})$.
 Put ${x}_{i} = t \restrict \hgt({z}_{i})$ and define ${D}_{i + 1}$ as
 \begin{align*}
  \{{s}_{i + 1} \in \Sa: \exists {u}_{i} \geq {x}_{i}\[{s}_{i + 1} \ \text{is according to} \ \Sigma(s) \ \text{at} \ {s}_{0}, \langle {x}_{0}, {u}_{0} \rangle, \dotsc, {s}_{i}, \langle {x}_{i}, {u}_{i} \rangle\]\}.
 \end{align*}
 ${D}_{i + 1} \in {N}^{\ast}_{l}$.
 Applying Lemma \ref{lem:submain} choose ${s}_{i + 1} \in {N}^{\ast}_{l} \cap {\pred}_{\Sa}(t)$ and ${x}_{i} \leq {u}_{i} \leq {s}_{i + 1}$ such that ${s}_{i + 1}$ is according to $\Sigma(s)$ at ${s}_{0}, \langle {x}_{0}, {u}_{0} \rangle, \dotsc, {s}_{i}, \langle {x}_{i}, {u}_{i} \rangle$.
 Note that if $i + 1 \leq m$, then since ${z}_{i + 1} \notin {N}^{\ast}_{l}$, $\hgt({z}_{i + 1}) \geq {N}^{\ast}_{l} \cap {\omega}_{1} > \hgt({s}_{i + 1})$, so that the construction can be continued, while if $i + 1 = m + 1$, then ${s}_{i + 1} \leq t$, so that ${x}_{m + 1} = {u}_{m + 1} = t$ is a permissible last move for II.
 Now, it is clear that ${s}_{0}, \langle {x}_{0}, {u}_{0} \rangle, \dotsc, {s}_{m + 1}, \langle {x}_{m + 1}, {u}_{m + 1}\rangle$ is a run of $\G(s)$ according to $\Sigma(s)$.
 However, if we let ${N}_{l} = {N}^{\ast}_{l}$, for each $0 \leq l \leq k$ and ${v}_{i} = {z}_{i}$ for all $i \in {I}_{1}$, then it is clear that (6)-(8) are satisfied.
 So II wins this run of $\G(s)$, contradicting that $\Sigma(s)$ is a winning strategy for I.
 \end{proof}
 Using Lemma 8, fix $s \in M \cap {\pred}_{\Sa}(t)$ such that $s \geq {t}^{\ast}$ and II wins $\G(s)$.
 Let $\Sigma(s) \in M$ be a winning strategy for II.
 For a fixed $i \in {I}_{0}$ consider the following statement:
 \begin{align*}
  &\text{if for each} \ j < i, {s}_{j}, {x}_{j}, {u}_{j} \in M \cap {\pred}_{\Sa}(t) \ \text{are given such that they form} \\ &\text{a partial run of} \ \G(s) \ \text{according to} \ \Sigma(s) \text{, then there exists a continuation} \tag{${\ast}_{i}$} \\ &{s}_{i}, {x}_{i}, {u}_{i} \in M \cap {\pred}_{\Sa}(t) \ \text{according to} \ \Sigma(s) \ \text{of this partial run such that} \\ &\text{there is no} \ 0 \leq l \leq m \ \text{so that} \ c(\{{\phi}_{i}({x}_{i}), {z}_{l}\}) \ \text{is defined and is equal to} \ 1.
 \end{align*}
Assume for a moment that (${\ast}_{i}$) holds for all $i \in {I}_{0}$.
Then using Lemma \ref{lem:submain} it is possible to choose a run ${s}_{0}, \langle {x}_{0}, {u}_{0} \rangle, \dotsc {s}_{m + 1}, \langle {x}_{m + 1}, {u}_{m + 1}\rangle$ of $\G(s)$ according to $\Sigma(s)$ such that for each $0 \leq i \leq m + 1$, ${s}_{i}, {x}_{i}, {u}_{i} \in M \cap {\pred}_{\Sa}(t)$ and for each $i \in {I}_{0}$ and $0 \leq l \leq m$, if $c(\{{\phi}_{i}({x}_{i}), {z}_{l}\})$ is defined, then it is equal to 0.
As this run lies in $M$, choose $\{{N}_{0}, \dotsc, {N}_{k}\} \in M$ and $\{{v}_{i}: i \in {I}_{1}\} \in M$ such that (6)-(8) are satisfied.
Put ${F}_{r} = {F}_{{q}_{0}} \cup \{{v}_{i}: i \in {I}_{1}\} \cup \{{\phi}_{i}({x}_{i}): i \in {I}_{0}\}$ and ${\N}_{r} = {\N}_{{q}_{0}} \cup \{{N}_{0}, \dotsc, {N}_{k}\}$.
Then $r = \langle {F}_{r}, {\N}_{r} \rangle \in M$ and $\langle r, {x}_{m + 1}\rangle \in M \cap D$.
Moreover, note that $\forall i \in {I}_{1} \forall 0 \leq l \leq m \[{v}_{i} \not\leq {z}_{l}\]$, and that for any $i \in {I}_{0}$ and $0 \leq l \leq m$, if ${\phi}_{i}({x}_{i}) \leq {z}_{l}$, then $c(\{{\phi}_{i}({x}_{i}), {z}_{l}\}) = 0$.
Therefore, $\langle \langle {F}_{r} \cup \{{z}_{0}, \dotsc, {z}_{m}\}, {\N}_{r} \cup \{{N}^{\ast}_{0}, \dotsc, {N}^{\ast}_{k}\} \rangle, t\rangle$ is a common extension of $\langle r, {x}_{m + 1} \rangle$ and $\langle q, t \rangle$.
Since this contradicts the hypothesis that no member of $D \cap M$ is compatible with $\langle q, t \rangle$, there must exist some $i \in {I}_{0}$ for which (${\ast}_{i}$) fails.
Fix $i \in {I}_{0}$ and ${s}_{j}, {x}_{j}, {u}_{j} \in M \cap {\pred}_{\Sa}(t)$ for $j < i$ witnessing this.
Define $u$ as follows.
If $i = 0$, then $u = s$, else $u = {u}_{i - 1}$.
In either case, $u \in M \cap {\pred}_{\Sa}(t)$ and that ${t}^{\ast} \leq u$.
Write $v = {\phi}_{i}(u)$ and note that ${y}_{i} \leq v \leq {z}_{i}$.
Let $E = \{{x}_{i}: \exists {s}_{i} \exists {u}_{i}\[{s}_{0}, \langle {x}_{0}, {u}_{0} \rangle, \dotsc, {s}_{i}, \langle {x}_{i}, {u}_{i} \rangle \ \text{is a partial run of} \ \G(s) \ \text{according to} \ \Sigma(s)\]\}$.
Define $Y$ to be $\{{\phi}_{i}({x}_{i}): {x}_{i} \in E\}$.
Note that $E, Y \in M$.
Since $\Sigma(s)$ is winning for II, $Y \subset X$.
It is easy to see that $Y$ is dense above $v$ in $\Sa$.
Indeed, let $w \in \Sa$ with $w \geq v$.
Then ${\phi}^{-1}_{i}(w) \geq u$, and so is a legitimate $i$th move for I.
Hence there exist ${x}_{i}, {u}_{i}$ such that ${s}_{0}, \langle {x}_{0}, {u}_{0} \rangle, \dotsc, {\phi}^{-1}_{i}(w), \langle {x}_{i}, {u}_{i} \rangle$ is a partial run of $G(s)$ according to $\Sigma(s)$.
Hence ${x}_{i} \in E$ and ${x}_{i} \geq {\phi}^{-1}_{i}(w)$, whence ${\phi}_{i}({x}_{i}) \in Y$ and ${\phi}_{i}({x}_{i}) \geq w$.
There is a function $g: Y \rightarrow \Sa$ in $M$ such that for all $x \in Y$, there exists ${s}_{i}$ such that ${s}_{0}, \langle {x}_{0}, {u}_{0} \rangle, \dotsc, {s}_{i}, \langle {\phi}^{-1}_{i}(x), {\phi}^{-1}_{i}(g(x)) \rangle$ is a partial run of $\G(s)$ according to $\Sigma(s)$.
Clearly, $g(x) \geq x$.
Now we
\begin{Claim} \label{claim:2}
For each $\alpha < {\omega}_{1}$, there exists ${F}_{\alpha}$ satisfying (1) and (3) such that $\bigwedge {F}_{\alpha} \geq v$.
\end{Claim}
\begin{proof}
If not, then there is $\alpha \in M$ witnessing this.
Fix such $\alpha \in M$.
Let ${F}_{\alpha} = \{{z}_{j} \in F: {y}_{i} \leq {z}_{j}\}$.
It is easy to see that $\bigwedge {F}_{\alpha} \notin M$, and hence $\hgt(\bigwedge {F}_{\alpha}) \geq \alpha$.
It is also easy to see that $\bigwedge {F}_{\alpha} \geq v$.
Now suppose that $x \in {\pred}_{Y}(\bigwedge {F}_{\alpha}) \cap {\Sa}_{< \alpha}$, and assume that $g(x) \leq \bigwedge{F}_{\alpha}$.
As every element of $Y$ is above ${y}_{i}$, this implies that ${\phi}^{-1}_{i}(g(x)), {\phi}^{-1}_{i}(x) \in M \cap {\pred}_{\Sa}(t)$ and that there is ${s}_{i} \in M \cap {\pred}_{\Sa}(t)$ such that ${s}_{0}, \langle {x}_{0}, {u}_{0} \rangle, \dotsc, {s}_{i}, \langle {\phi}^{-1}_{i}(x), {\phi}^{-1}_{i}(g(x)) \rangle$ is a partial run of $\G(s)$ according to $\Sigma(s)$.
By the hypothesis that (${\ast}_{i}$) fails, there is ${z}_{l} \in F$ such that $c(\{x, {z}_{l}\}) = 1$.
In particular, ${y}_{i} \leq x \leq {z}_{l}$.
So ${z}_{l} \in {F}_{\alpha}$, and (3) is satisfied.
\end{proof}
To complete the proof of the lemma, choose ${F}_{\alpha}$ for each $\alpha < {\omega}_{1}$ as in the claim. Then $Z = \{\bigwedge {F}_{\alpha}: \alpha < {\omega}_{1}\}$ is an uncountable subset of $\Sa$.
Find $z \in Z$ such that $Z$ is dense above $z$ in $\Sa$.
Since $z \geq v$ and $Y$ is dense above $v$ in $\Sa$, it is possible to choose ${x}_{0} \in Y \subset X$  with ${x}_{0} \geq z$.
Now both $Z$ and $Y$ are dense above ${x}_{0}$ in $\Sa$.
\end{proof}
 \begin{Def} \label{def:sigma}
Let $R$ be a Suslin tree and fix $c: {R}^{[2]} \rightarrow 2$.
Fix a $C$-sequence $\langle {c}_{\alpha}: \alpha < {\omega}_{1}\rangle$ such
that if $\beta > 1$, then $\forall n \in \omega \[{c}_{\beta}(n) > 0\]$.
For $0 < \beta < {\omega}_{1}$, $t \in R$, and $n \in \omega$ define $L(\beta,
t, n)$ to be the set of all $A$ such that
\begin{enumerate}
 \item
 $\exists s \in R\[s < t \ \text{and} \ A \subset {\pred}_{R}(s)\]$.
 \item
 $\otp(A) = {\omega}^{{c}_{\beta}(n)}$.
 \item
 ${A}^{[2]} \subset {K}_{1}$.
 \item
 $\{u \in {\cone}_{R}(t): A \otimes \{u\} \subset {K}_{1}\}$ is uncountable.
\end{enumerate}
\end{Def}
Note that no member of $L(\beta, t, n)$ is empty.
Next, if $B \in L(\beta, t, n + 1)$ and $A \subset B$ with $\otp(A) = {\omega}^{{c}_{\beta}(n)}$, then $A \in L(\beta, t, n)$.
Moreover, if $t \leq u$, $A \in L(\beta, u, n)$, and $\exists s \in R\[s < t \ \text{and} \ A \subset {\pred}_{R}(s)\]$, then $A \in L(\beta, t, n)$.
Also if $t \leq u$, $A \in L(\beta, t, n)$, and $\{v \in {\cone}_{R}(u): A \otimes \{v\} \subset {K}_{1}\}$ is uncountable, then $A \in L(\beta, u, n)$.
\begin{Def} \label{def:sigma1}
Fix a well-ordering of $\Pset(R)$, say $\wo$.
For $A, B \subset R$ and $t \in R$  we say that \emph{$B$ follows $A$ with respect to $t$} if $\forall a \in A \forall b \in B\[a < b\]$, $A \otimes B \subset {K}_{1}$, and $\{u \in {\cone}_{R}(t): \left( A \cup B \right) \otimes \{u\} \subset {K}_{1}\}$ is uncountable.
It is clear that if $B$ follows $A$ with respect to $t$ and $C \subset B$, then $C$ follows $A$ with respect to $t$.
Also, if $t \leq u$ and $B$ follows $A$ with respect to $u$, then $B$ follows $A$ with respect to $t$.
If $t \leq u$, $B$ follows $A$ with respect to $t$, and $\{v \in {\cone}_{R}(u): \left( A \cup B \right) \otimes \{v\} \subset {K}_{1}\}$ is uncountable, then $B$ follows $A$ with respect to $u$.
For $t \in R$ and $0 < \beta < {\omega}_{1}$, define a function ${\sigma}_{\beta, t}: \omega \rightarrow \Pset(R)$ as follows.
Fix $n \in \omega$ and suppose that for all $m < n$, ${\sigma}_{\beta, t}(m)$ has been defined.
Put ${A}_{\beta, t, n} = {\bigcup}_{m < n}{{\sigma}_{\beta, t}(m)} \subset R$.
Consider
\begin{align*}
 \left\{B \in L(\beta, t, n): B \ \text{follows} \ {A}_{\beta, t, n} \ \text{with respect to} \ t \right\}.
\end{align*}
If this set is empty, then set ${\sigma}_{\beta, t}(n) = 0$.
Otherwise, set ${\sigma}_{\beta, t}(n)$ to be the $\wo$-least element of this set.
Define ${A}_{\beta, t} = {\bigcup}_{n \in \omega}{{\sigma}_{\beta, t}(n)}$.
\end{Def}
It is clear that for each $n \in \omega$, either ${\sigma}_{\beta, t}(n) = 0$ or ${\sigma}_{\beta, t}(n) \in L(\beta, t, n)$, but not both.
In either case, observe that ${\sigma}_{\beta, t}(n) \subset {\pred}_{R}(t)$, that $\forall a \in {A}_{\beta, t, n} \forall b \in {\sigma}_{\beta, t}(n)\[a < b\]$, and that $\left( {A}_{\beta, t, n} \otimes {\sigma}_{\beta, t} (n) \right) \cup {\left({\sigma}_{\beta, t}(n)\right)}^{\[2\]} \subset {K}_{1}$.
Therefore, if $\forall n \in \omega \[{\sigma}_{\beta, t}(n) \in L(\beta, t, n)\]$, then ${A}_{\beta, t}$ is a subset of ${\pred}_{R}(t)$ of order type ${\omega}^{\beta}$ such that ${A}^{\[2\]}_{\beta, t} \subset {K}_{1}$.
Furthermore, for all $t \in R \setminus \{\min(R)\}$ and all $n \in \omega$, there is $s \in R$ with $s < t$ such that ${A}_{\beta, t, n} \subset {\pred}_{R}(s)$.
\begin{Lemma} \label{lem:sigma1}
For any $t \in R$ and $0 < \beta < {\omega}_{1}$, if there exists $m \in \omega$ such that ${\sigma}_{\beta, t}(m) = 0$, then $\forall n \geq m \[{\sigma}_{\beta, t}(n) = 0\]$.
\end{Lemma}
\begin{proof}
Prove by induction on $n \geq m$ that ${\sigma}_{\beta, t}(n) = 0$.
When $n = m$, this is the hypothesis.
Suppose this is true for $n \geq m$.
Then ${A}_{\beta, t, n + 1} = {A}_{\beta, t, n} \cup {\sigma}_{\beta, t}(n) = {A}_{\beta, t, n}$.
If ${\sigma}_{\beta, t}(n + 1) \neq 0$, then ${\sigma}_{\beta, t}(n + 1) \in L(\beta, t, n + 1)$, and ${\sigma}_{\beta, t}(n + 1)$ follows ${A}_{\beta, t, n}$ with respect to $t$.
In particular, $\otp({\sigma}_{\beta, t}(n + 1)) = {\omega}^{{c}_{\beta}(n + 1)}$.
Choose $B \subset {\sigma}_{\beta, t}(n + 1)$ with $\otp(B) = {\omega}^{{c}_{\beta}(n)}$.
Then $B \in L(\beta, t, n)$ and $B$ follows ${A}_{\beta, t, n}$ with respect to $t$.
This contradicts the fact that ${\sigma}_{\beta, t}(n) = 0$.
\end{proof}
\begin{Def} \label{def:sigma2}
 For $t \in R$ and $0 < \beta < {\omega}_{1}$, if there is $n \in \omega$ such that ${\sigma}_{\beta, t}(n) = 0$, then let ${n}_{\beta, t}$ be the least such $n$.
 Otherwise, put ${n}_{\beta, t} = \omega$.
 Define ${X}_{\beta, t} = \{u \in {\cone}_{R}(t): {A}_{\beta, t} \otimes \{u\} \subset {K}_{1}\}$.
\end{Def}
Note that ${A}_{\beta, t} = {\bigcup}_{n < {n}_{\beta, t}}{{\sigma}_{\beta, t}(n)}$.
Observe also that if ${n}_{\beta, t} < \omega$, then ${X}_{\beta, t}$ is uncountable.
\begin{Lemma} \label{lem:sigma}
 Fix $t, u \in R$ with $t \leq u$, and $0 < \beta < {\omega}_{1}$.
 Suppose that ${X}_{\beta, t} \cap {\cone}_{R}(u)$ is uncountable.
 Moreover assume that $\exists s \in R \[s < t \ \text{and} \ {A}_{\beta, u} \subset {\pred}_{R}(s)\]$.
 Then ${\sigma}_{\beta, t} = {\sigma}_{\beta, u}$.
\end{Lemma}
\begin{proof}
First a preliminary observation: It follows from the hypotheses that for each $n \in \omega$, $\{v \in {\cone}_{R}(u): {\sigma}_{\beta, t}(n) \otimes \{v\} \subset {K}_{1}\}$ and $\{v \in {\cone}_{R}(u): {A}_{\beta, t, n} \otimes \{v\} \subset {K}_{1}\}$ are both uncountable.
Now suppose for a contradiction that there exists $n \in \omega$ such that ${\sigma}_{\beta, t}(n) \neq {\sigma}_{\beta, u}(n)$, and choose the minimal $n \in \omega$ with this property.
Then ${A}_{\beta, t, n} = {A}_{\beta, u, n}$.
Assume that ${\sigma}_{\beta, t}(n) \neq 0$.
Thus ${\sigma}_{\beta, t}(n) \in L(\beta, t, n)$ and ${\sigma}_{\beta, t}(n)$ follows ${A}_{\beta, t, n}$ with respect to $t$.
Since $\{v \in {\cone}_{R}(u): {\sigma}_{\beta, t}(n) \otimes \{v\} \subset {K}_{1}\}$ is uncountable, ${\sigma}_{\beta, t}(n) \in L(\beta, u, n)$.
Also, since $\{v \in {\cone}_{R}(u): \left( {A}_{\beta, t, n} \cup {\sigma}_{\beta, t}(n) \right) \otimes \{v\} \subset {K}_{1}\} = \{v \in {\cone}_{R}(u): {A}_{\beta, t, n + 1} \otimes \{v\} \subset {K}_{1}\}$ is uncountable, ${\sigma}_{\beta, t}(n)$ follows ${A}_{\beta, u, n}$ with respect to $u$.
It follows that ${\sigma}_{\beta, u}(n) \neq 0$ and ${\sigma}_{\beta, u}(n) \wo {\sigma}_{\beta, t}(n)$.

Next suppose that ${\sigma}_{\beta, u}(n) \neq 0$.
Then ${\sigma}_{\beta, u}(n) \in L(\beta, u, n)$ and ${\sigma}_{\beta, u}(n)$ follows ${A}_{\beta, u, n}$ with respect to $u$.
Note that there is $s \in R$ such that $s < t$ and ${\sigma}_{\beta, u}(n) \subset {A}_{\beta, u} \subset {\pred}_{R}(s)$.
Therefore, ${\sigma}_{\beta, u}(n) \in L(\beta, t, n)$.
Also, ${\sigma}_{\beta, u}(n)$ follows ${A}_{\beta, t, n}$ with respect to $t$.
Thus ${\sigma}_{\beta, t}(n) \neq 0$ and ${\sigma}_{\beta, t}(n) \wo {\sigma}_{\beta, u}(n)$.
However, these two implication that we have established imply that both ${\sigma}_{\beta, t}(n)$ and ${\sigma}_{\beta, u}(n)$ are equal to $0$, a contradiction.
\end{proof}
Next is a lemma that is of independent interest and can be considered as a part of set theory folklore.
It asserts the existence of certain types of ultrafilters on countable indecomposable ordinals under the hypothesis $\p > {\omega}_{1}$.
This lemma also plays a similar role in Todorcevic's proof that $\PID + \p > {\omega}_{1}$ implies ${\omega}_{1} \rightarrow {({\omega}_{1}, \alpha)}^{2}$.
\begin{Lemma} \label{lem:ultrafilter}
 Assume $\p >  {\omega}_{1}$.
 For each $0 < \beta < {\omega}_{1}$ and a well ordered set $X = \langle X,
{<}_{X} \rangle$ of order type ${\omega}^{\beta}$ there is an ultrafilter
${\U}_{\beta}(X)$ on $X$ such that
 \begin{enumerate}
  \item
   For each $A \in {\U}_{\beta}(X)$, $\otp(A) = {\omega}^{\beta}$.
  \item
   For any $\F \subset {\U}_{\beta}(X)$ of size at most ${\omega}_{1}$, there is
$Y \subset X$ such that $\otp(Y) = {\omega}^{\beta}$ and $\forall Z \in \F \[Y \setminus Z \ \text{is a bounded subset of} \ X\]$. Moreover if $\lc \F \rc =
{\omega}_{1}$, then there exists $\GG \in {\[\F\]}^{{\omega}_{1}}$ such that
$\otp\left(\bigcap \GG \right) = {\omega}^{\beta}$.
 \end{enumerate}
\end{Lemma}
\begin{proof}
The proof is by induction on $\beta$.
If $\beta = 1$, then $X$ has order type $\omega$, and we can let
${\U}_{\beta}(X)$ be any ultrafilter on $X$.
It is clear that (1) is satisfied.
For (2), fix $\F \subset {\U}_{\beta}(X)$ with $\lc \F \rc \leq {\omega}_{1}$.
As $\p > {\omega}_{1}$, there is $Y \in {\[X\]}^{\omega}$ such that $\forall Z
\in \F \[Y \setminus Z \ \text{is finite}\]$.
It is clear that $\otp(Y) = \omega$ and $\forall Z \in \F \[Y \setminus Z \ \text{is bounded in} \ X\]$.
Next, if $\lc \F \rc = {\omega}_{1}$ and $Y \subset X$ is as above, then there
is a finite $F \in {\[Y\]}^{< \omega}$ and $\GG \in {\[\F\]}^{{\omega}_{1}}$
such that $\forall Z \in \GG \[Y \setminus F \subset Z\]$.
As $\otp(Y \setminus F) = \omega$, it is clear that $\otp(\bigcap \GG) =
\omega$.

Next suppose $\beta > 1$.
Fix $\langle {X}_{n}: n \in \omega \rangle$ such that $X = {\bigcup}_{n \in
\omega}{{X}_{n}}$ and $\forall n \in \omega \[\otp({X}_{n}) =
{\omega}^{{c}_{\beta}(n)} \wedge {X}_{n} \; {<}_{X} \; {X}_{n + 1}\]$.
Here ${X}_{n} \; {<}_{X} \; {X}_{n + 1}$ means $\forall x \in {X}_{n} \forall y
\in {X}_{n + 1}\[x \; {<}_{X} \; y\]$.
Assume that ${\U}_{{c}_{\beta}(n)}({X}_{n})$ has been constructed for all $n \in
\omega$.
Put ${\U}_{\beta}(X) = \{A \subset X: \{n \in \omega: A \cap {X}_{n} \in
{\U}_{{c}_{\beta}(n)}({X}_{n})\} \in {\U}_{1}(\omega)\}$.
This is clearly an ultrafilter on $X$.
It is also clear that (1) is satisfied by ${\U}_{\beta}(X)$.
For (2), fix $\F \subset {\U}_{\beta}(X)$ of size at most ${\omega}_{1}$.
For each $n \in \omega$, put
\begin{align*}
 {\F}_{n} = \{Z \cap {X}_{n}: Z \in \F \wedge Z \cap {X}_{n} \in
{\U}_{{c}_{\beta}(n)}({X}_{n})\}.
 \end{align*}
${\F}_{n} \subset {\U}_{{c}_{\beta}(n)}({X}_{n})$ and $\lc {\F}_{n} \rc \leq
{\omega}_{1}$.
So choose ${Y}_{n} \subset {X}_{n}$ such that $\otp({Y}_{n}) =
{\omega}^{{c}_{\beta}(n)}$ and $\forall A \in {\F}_{n}\[{Y}_{n} \setminus A \ \text{is a bounded subset of} \ {X}_{n}\]$.
For each $Z \in \F$ define a function ${S}_{Z}$ with domain $\omega$ as follows.
Given $n \in \omega$, if $Z \cap {X}_{n} \in {\U}_{{c}_{\beta}(n)}({X}_{n})$,
then ${S}_{Z}(n) = {Y}_{n} \setminus \left( Z \cap {X}_{n} \right)$.
Otherwise ${S}_{Z}(n) = 0$.
In either case ${S}_{Z}(n)$ is a bounded subset of ${X}_{n}$.
Use the fact that $\b < {\omega}_{1}$ to choose $\langle {B}_{n}: n \in \omega
\rangle$ such that for each $\forall n \in \omega \[{B}_{n} \ \text{is a bounded
subset of} \ {X}_{n}\]$ and $\forall Z \in \F \forallbutfin n \in \omega
\[{S}_{Z}(n) \subset {B}_{n}\]$.
Note that $\otp({Y}_{n} \setminus {B}_{n}) = {\omega}^{{c}_{\beta}(n)}$, for all
$n \in \omega$.
For each $Z \in \F$, put $\dom(Z) = \{n \in \omega: Z \cap {X}_{n} \in
{\U}_{{c}_{\beta}(n)}({X}_{n})\}$.
$\{\dom(Z): Z \in \F\}$ is a subset of ${\U}_{1}(\omega)$ of size at most
${\omega}_{1}$.
Choose $D \in \cube$ such that $\forall Z \in \F \[D \setminus \dom(Z) \
\text{is finite}\]$.
Put $Y = {\bigcup}_{n \in D}{\left({Y}_{n} \setminus {B}_{n}\right)}$.
It is clear that $Y \subset X$ and $\otp(Y) = {\omega}^{\beta}$.
Fix $Z \in \F$ and fix ${n}_{Z}\in \omega$ such that $D \setminus {n}_{Z}
\subset \dom(Z)$ and $\forall n \geq {n}_{Z}\[{S}_{Z}(n) \subset {B}_{n}\]$.
We claim that $Y \setminus Z$ is bounded by $\min\left({X}_{{n}_{Z}}\right)$.
Indeed fix $y \in Y \setminus Z$.
Then $y \in {Y}_{n} \setminus {B}_{n}$ for some $n \in D$.
If $n \geq {n}_{Z}$, then $n \in \dom(Z)$ and ${S}_{Z}(n) \subset {B}_{n}$.
So $Z \cap {X}_{n} \in {\U}_{{c}_{\beta}(n)}({X}_{n})$, and so ${S}_{Z}(n) =
{Y}_{n} \setminus \left( Z \cap {X}_{n} \right)$.
As $y \in {Y}_{n}$ and $y \notin Z$, we have that $y \in {S}_{Z}(n)$.
However, since $y \notin {B}_{n}$, this contradicts ${S}_{Z}(n) \subset
{B}_{n}$.
Therefore, $n < {n}_{z}$ and since $y \in {Y}_{n} \subset {X}_{n}$, we conclude
that $y \; {<}_{X} \; \min\left( {X}_{{n}_{Z}}\right)$.

Now, if $\lc \F \rc = {\omega}_{1}$, then there exist $n \in \omega$ and $\GG \in {\[\F\]}^{{\omega}_{1}}$ such that $\forall Z \in \GG\[{n}_{Z} = n\]$.
Note that $\otp\left(\left\{y \in Y: y \; {<}_{X} \; \min\left( {X}_{n} \right) \right\} \right) < {\omega}^{\beta}$.
Therefore, ${Y}^{\ast} = Y \setminus \left\{y \in Y: y \; {<}_{X} \; \min\left(
{X}_{n} \right) \right\}$ has order type ${\omega}^{\beta}$.
By what has been proved above, ${Y}^{\ast} \subset \bigcap \GG$, whence
$\otp\left( \bigcap \GG \right) = {\omega}^{\beta}$.
\end{proof}
We can now finish the proof of Theorem \ref{thm:main}.
\begin{proof}[Proof of Theorem \ref{thm:main}]
 We show by induction on $\beta < {\omega}_{1}$ that for any Suslin tree $T \subset \Sa$ and any $c:{T}^{\[2\]} \rightarrow 2$, either there exists $Y \in {\[T\]}^{{\omega}_{1}}$ and $g: Y \rightarrow \Sa$ such that $\forall y \in Y\[g(y) \geq y\]$ and ${Y}^{\[2\]}_{g} \subset {K}_{0, c}$, or for each $X \in {\[T\]}^{{\omega}_{1}}$ there exist $x \in X$, $B \subset {\pred}_{X}(x)$, and $Z \in {\[{\cone}_{X}(x)\]}^{{\omega}_{1}}$ such that $\otp(B) = {\omega}^{\beta}$ and ${B}^{\[2\]} \cup B \otimes Z \subset {K}_{1, c}$.
 This is sufficient to imply Theorem \ref{thm:main}, for given $S \in {\[\Sa\]}^{{\omega}_{1}}$ and $c: {S}^{\[2\]} \rightarrow 2$, let $T \in {\[S\]}^{{\omega}_{1}}$ be a Suslin tree.
 Suppose the first alternative of Theorem \ref{thm:main} fails and let $\alpha < {\omega}_{1}$ be given.
 Choose $\beta < {\omega}_{1}$ such that $\alpha \leq {\omega}^{\beta}$.
 Applying the above statement to $\beta$, $T$, and $c \restrict {T}^{\[2\]}$ with $X = T$, we can get $x \in T$ and $B \subset {\pred}_{S}(x)$ such that $\otp(B) = {\omega}^{\beta}$ and ${B}^{\[2\]} \subset {K}_{1, c}$.
 Taking ${B}^{\ast} \subset B$ with $\otp({B}^{\ast}) = \alpha$, we get what we want.

 To prove the above statement, fix $\beta < {\omega}_{1}$ and assume that the statement holds for all smaller ordinals.
 Let $T$ and $c$ be given and suppose that the first alternative of the statement fails.
 In particular, this implies that there is no $X \in {\[T\]}^{{\omega}_{1}}$, such that $\P(X, c \restrict {X}^{\[2\]})$ is proper and preserves $\Sa$.
 For if not, let ${D}_{\alpha} = \{q \in \P(X, c \restrict {X}^{\[2\]}): \exists t \in {F}_{q}\[\hgt(t) > \alpha\]\}$.
 By Lemma \ref{lem:density}, ${D}_{\alpha}$ is dense in $\P(X, c \restrict {X}^{\[2\]})$.
 Applying $\PFA(\Sa)$, let $G$ be a filter on $\P(X, c \restrict {X}^{\[2\]})$ such that $\forall \alpha < {\omega}_{1} \[G \cap {D}_{\alpha} \neq 0\]$.
 Let $Y = {\bigcup}_{q \in G}{{F}_{q}}$.
 Then $Y \in {\[T\]}^{{\omega}_{1}}$.
 If we let $g: Y \rightarrow \Sa$ be defined by $g(y) = y$, for all $y \in Y$, then ${Y}^{\[2\]}_{g} = {Y}^{\[2\]} \subset {K}_{0, c}$, contradicting the hypothesis that the first alternative of the statement fails.

 Now, fix any $X \in {\[T\]}^{{\omega}_{1}}$.
 By Lemma \ref{lem:main1} there exist ${x}_{0} \in X$, $Y \in {\[X\]}^{{\omega}_{1}}$, a sequence $\langle {F}_{\alpha}: \alpha < {\omega}_{1}\rangle$, and a function $g: Y \rightarrow \Sa$ satisfying (1)-(3) of Lemma \ref{lem:main1}.
 If $\beta = 0$, then fix $y \in Y$ with $y \geq {x}_{0}$, and put $x = y$ and $B = \{y\}$.
 Then for each $\alpha < {\omega}_{1}$, if $\hgt(y) < \alpha$ and $\bigwedge {F}_{\alpha} \geq g(y)$, then there is ${t}_{\alpha} \in {F}_{\alpha}$ such that $c(\{y, {t}_{\alpha}\}) = 1$.
 Letting $Z = \{{t}_{\alpha}: \alpha > \hgt(y) \ \text{and} \ \bigwedge {F}_{\alpha} \geq  g(y) \}$, we get what we want.

 Assume now that $0 < \beta < {\omega}_{1}$.
 Let $\bar{Y} = g''Y$.
 Fix $R \subset {\cone}_{\bar{Y}}({x}_{0})$, a Suslin tree.
 For each $s \in R$ choose ${y}_{s} \in Y$ such that $g({y}_{s}) = s$.
 Observe that if $s < t$, then ${y}_{s}$ and ${y}_{t}$ are comparable and different.
 Define $d: {R}^{\[2\]} \rightarrow 2$ by $d(\{s, t\}) = 1$ iff ${y}_{s} < {y}_{t}$ and $c(\{{y}_{s}, {y}_{t}\}) = 1$, for any $s, t \in R$ with $s < t$.
 We claim that it is enough to find $u \in R$ and $\bar{B} \subset {\pred}_{R}(u)$ with $\otp(\bar{B}) = {\omega}^{\beta}$ such that ${\bar{B}}^{\[2\]} \subset {K}_{1, d}$.
 Indeed, suppose this can be done.
 Choose any $x \in Y$ with $x \geq u$
 Let ${B}^{\ast} = \{{y}_{s}: s \in \bar{B}\}$.
 For any $s, t \in \bar{B}$ with $s < t$, ${y}_{s} < {y}_{t}$ because $d(\{s, t\}) = 1$.
 So $\otp({B}^{\ast}) = {\omega}^{\beta}$.
 Also it is clear that ${B}^{\ast} \subset {\pred}_{X}(x)$.
 If $\alpha < {\omega}_{1}$ is such that $\alpha > \hgt(x)$ and $\bigwedge {F}_{\alpha} \geq x$, then for any $y \in {B}^{\ast}$, $g(y) \leq \bigwedge {F}_{\alpha}$, and so there is $t \in {F}_{\alpha}$ such that $c(\{y, t\}) = 1$.
 Therefore, letting ${\U}_{\beta}({B}^{\ast})$ be as in Lemma \ref{lem:ultrafilter} (note that $\PFA(\Sa)$ implies $\p > {\omega}_{1}$; so Lemma \ref{lem:ultrafilter} may be applied) and letting $I = \{\alpha < {\omega}_{1}: \alpha > \hgt(x) \ \text{and} \ \bigwedge {F}_{\alpha} \geq x\}$, for each $\alpha \in I$, there is ${Y}_{\alpha} \in {\U}_{\beta}({B}^{\ast})$ and ${t}_{\alpha} \in {F}_{\alpha}$ such that $\forall y \in {Y}_{\alpha}\[c(\{y, {t}_{\alpha}\}) = 1\]$.
 There exists $J \in {\[I\]}^{{\omega}_{1}}$ such that $\otp\left({\bigcap}_{\alpha \in J}{{Y}_{\alpha}}\right) = {\omega}^{\beta}$.
 It is clear that $B = {\bigcap}_{\alpha \in J}{{Y}_{\alpha}}$, $Z = \{{t}_{\alpha}: \alpha \in J\}$, and $x$ are as needed.

 Thus we may concentrate on finding $u \in R$ and $\bar{B} \subset {\pred}_{R}(u)$ as above.
 We will apply the notation of Definitions \ref{def:sigma}, \ref{def:sigma1}, and \ref{def:sigma2}, and Lemmas \ref{lem:sigma} and \ref{lem:sigma1} to $R$, $d$, and $\beta$.
 If there exists $u \in R$ such that ${n}_{\beta, u} = \omega$, then letting $\bar{B} = {A}_{\beta, u}$ works.
 Thus assume that for each $u \in R$, ${n}_{\beta, u} < \omega$.
 Then for each $u \in R \setminus \{\min(R)\}$, there is $f(u) \in R$ such that $f(u) < u$ and ${A}_{\beta, u} \subset {\pred}_{R}(f(u))$.
 So by Lemma \ref{lem:regressive} and the pigeonhole principle there exist $U \in {\[R \setminus \{\min(R)\}\]}^{{\omega}_{1}}$, $s \in R$, and $n \in \omega$ such that $\forall u \in U \[f(u) = s \ \text{and} \ {n}_{\beta, u} = n\]$.
 Fix $x \in U$ such that $U$ is dense above $x$ in $\Sa$.
 Since ${n}_{\beta, x} < \omega$, ${X}_{\beta, x}$ is uncountable.
 Choose $u \in {X}_{\beta, x}$ such that ${X}_{\beta, x}$ is dense above $u$ in $\Sa$.
 Apply the inductive hypothesis to ${c}_{\beta}(n)$, $R$, and $d$.
 Suppose that the first alternative holds.
 Let $F \in {\[R\]}^{{\omega}_{1}}$ and ${g}^{\ast}: F \rightarrow \Sa$ be such that $\forall t \in F \[{g}^{\ast}(t) \geq t\]$ and ${F}^{\[2\]}_{{g}^{\ast}} \subset {K}_{0, d}$.
 Let ${Y}^{\ast} = \{{y}_{t}: t \in F\}$ and define $h: {Y}^{\ast} \rightarrow \Sa$ by $h(y) = {g}^{\ast}(g(y))$, for each $y \in {Y}^{\ast}$.
 Then ${Y}^{\ast} \in {\[T\]}^{{\omega}_{1}}$, $\forall y \in {Y}^{\ast}\[h(y) \geq y\]$, and ${{Y}^{\ast}}^{\[2\]}_{h} \subset {K}_{0, c}$.
 This contradicts the hypothesis that the first alternative fails for $\beta$, $T$, and $c$.
 So the second alternative must hold for ${c}_{\beta}(n)$, $R$, and $d$.
 Since $V = {\cone}_{{X}_{\beta, x}}(u) \in {\[R\]}^{{\omega}_{1}}$, we can find $v \in V$, $B \subset {\pred}_{V}(v)$, and $W \in {\[{\cone}_{V}(v)\]}^{{\omega}_{1}}$ such that $\otp(B) = {\omega}^{{c}_{\beta}(n)}$, and ${B}^{\[2\]} \cup \left(B \otimes W \right) \subset {K}_{1, d}$.
 Fix $w \in W$ such that $W$ is dense above $w$ in $\Sa$.
 Choose $y \in U$ with $y > w$.
 Note that ${X}_{\beta, x} \cap {\cone}_{R}(y)$ is uncountable.
 Furthermore, as $y \in U$ $f(y) = s < x$ and ${A}_{\beta, y} \subset {\pred}_{R}(s)$.
 Therefore, Lemma \ref{lem:sigma} applies and implies that ${\sigma}_{\beta, x} =  {\sigma}_{\beta, y}$.
 In particular, ${A}_{\beta, x} = {A}_{\beta, y}$.
 Also, since $y \in U$, ${n}_{\beta, y} = n$.
 So ${A}_{\beta, y} = {A}_{\beta, y, n}$ and ${\sigma}_{\beta, y}(n) = 0$.
 However ${\cone}_{W}(y)$ is uncountable and ${\cone}_{W}(y) \subset \{z \in {\cone}_{R}(y): \left({A}_{\beta, y, n} \cup B \right) \otimes \{z\} \subset {K}_{1, d}\} \subset \{z \in {\cone}_{R}(y): B \otimes \{z\} \subset {K}_{1, d}\}$. So it is easy to check that $B \in L(\beta, y, n)$ and that $B$ follows ${A}_{\beta, y, n}$ with respect to $y$. However, this contradicts ${\sigma}_{\beta, y}(n) = 0$, finishing the proof.
\end{proof}
The reader may conjecture that the stronger form of Theorem \ref{thm:main} in which $g$ is always equal to the identity function holds.
However, our next counterexample shows that this is provably false in $\ZFC$.
This same negative partition relation for non-special trees of cardinality $\c$ with no uncountable chains and with the property that every subset of size $< \c$ is special was proved under the hypothesis that $\p = \c$ by Todorcevic in \cite{posetpartition}.
However, our result below is the first such negative partition relation known to be provable in $\ZFC$.
This result will finish this section.
\begin{Theorem} \label{thm:coloring}
 There is $c: {\Sa}^{\[2\]} \rightarrow 2$ such that
 \begin{enumerate}
  \item
  There is no $X \in {\[\Sa\]}^{{\omega}_{1}}$ such that ${X}^{\[2\]} \subset {K}_{0}$.
  \item
  There is no $s \in \Sa$ and $B \subset {\pred}_{\Sa}(s)$ such that $\otp(B) = \omega + 2$ and ${B}^{\[2\]} \subset {K}_{1}$.
 \end{enumerate}
\end{Theorem}
\begin{proof}
 Let ${\omega}^{\uparrow \omega}$ denote $\{f \in \BS: \forall n \in \omega \[f(n) < f(n + 1)\]\}$.
 For $f, g \in \BS$, if $f \neq g$, let $\Delta(f, g)$ denote the least $n \in \omega$ such that $f(n) \neq g(n)$.
 Choose a collection $\{{f}_{s}: s \in \Sa\}$ of ${\aleph}_{1}$-many pairwise distinct elements of ${\omega}^{\uparrow \omega}$.
 For each $s \in \Sa$, let $\{{s}^{+}_{n}: n \in \omega\}$ be a 1-1 enumeration of ${\succc}_{\Sa}(s)$.
 Recall that for each $s \in \Sa$, $\existsinf n \in \omega \[{s}^{\frown}{\langle n \rangle} \in \Sa\]$.
 Now, define $c: {\Sa}^{\[2\]} \rightarrow 2$ as follows.
 For any pair $s, t \in \Sa$, if $s < t$, then there is a unique $n \in \omega$ such that ${s}^{+}_{n} \leq t$.
 If ${f}_{t}(\Delta({f}_{s}, {f}_{t})) = n$, then set $c(\{s, t\}) = 1$.
 Otherwise $c(\{s, t\}) = 0$.
 The first claim will establish (1).
 \begin{Claim} \label{claim:coloring1}
  There is no $X \in {\[\Sa\]}^{{\omega}_{1}}$ such that ${X}^{\[2\]} \subset {K}_{0}$.
 \end{Claim}
 \begin{proof}
 It is possible to deduce this claim from Lemma 5.3 of \cite{partition}.
 However, we will give a self contained proof below.

 Suppose not.
 Fix a counterexample $X$.
 Let $\chi$ be a sufficiently large regular cardinal ($\chi$ could be the cardinal fixed in Definition \ref{def:poset}).
 Let $M \prec H(\chi)$ be countable with $\Sa, \langle {f}_{s}: s \in \Sa\rangle, \langle \langle {s}^{+}_{n}: n \in \omega \rangle: s \in \Sa \rangle, c, X \in M$.
 Let $x \in X \cap M$ be such that $X$ is dense above $x$ in $\Sa$.
 Fix $t \in X \setminus M$ with $x < t$, and fix $u \in X$ with $u > t$.
 We will get a contradiction if we can show that ${f}_{t} = {f}_{u}$.
 To this end, fix $m \in \omega$ and assume that ${f}_{t}(i) = {f}_{u}(i)$, for all $i < m$.
 Let $\sigma = {f}_{t} \restrict m = {f}_{u} \restrict m$.
 Consider any $s \in {\pred}_{X}(t) \cap M$.
 Let $n(s)$ denote the unique $n \in \omega$ such that ${s}^{+}_{n} \leq t$.
 If $\sigma = {f}_{s} \restrict m$ and if $n(s) = {f}_{t}(m)$, then since $c(\{s, t\}) = 0$, it follows that ${f}_{s}(m) = {f}_{t}(m)$.
 Put $n = {f}_{t}(m)$ and ${n}^{\ast} = {f}_{u}(m)$.
 Let
 \begin{align*}
  D = \left\{v \in \Sa: \exists s \in X \[{f}_{s} \restrict m = \sigma \ \text{and} \ {f}_{s}(m) = {n}^{\ast} \ \text{and} \ {s}^{+}_{n} = v\] \right\}.
 \end{align*}
 It is easy to see that $D \in M$.
 Let $L = {\pred}_{\Sa}(t)$.
 Note that ${u}^{+}_{n} \in D$.
 Therefore, $\forall y \in L \cap M \exists v \in D\[y \leq v\]$.
 So by Lemma \ref{lem:submain}, we can find $v \in L \cap M \cap D$.
 Let $s \in X$ be such that ${f}_{s} \restrict m = \sigma$, ${f}_{s}(m) = {n}^{\ast}$ and ${s}^{+}_{n} = v$.
 Then $s < t$ and $s \in M \cap {\pred}_{X}(t)$.
 Since ${s}^{+}_{n} = v \leq t$, $n(s) = n = {f}_{t}(m)$.
 Thus ${f}_{u}(m) = {n}^{\ast} = {f}_{s}(m) = {f}_{t}(m)$.
 So by induction on $m \in \omega$, $\forall m \in \omega \[{f}_{u}(m) = {f}_{t}(m)\]$, which is a contradiction.
 \end{proof}
 We next work toward showing that (2) holds.
 We need a few preliminary claims.
 Aiming for a contradiction, fix $s \in \Sa$ and $B \subset {\pred}_{\Sa}(s)$ such that $\otp(B) = \omega + 2$ and ${B}^{\[2\]} \subset {K}_{1}$.
 For each $i < \omega + 2$, let $s(i)$ denote the $i$-th element of $B$, and for each $i < \omega + 1$, let $n(i)$ be the unique $n \in \omega$ such that ${\left( s(i) \right)}^{+}_{n} \in {\pred}_{\Sa}(s)$.
 \begin{Claim} \label{claim:coloring2}
 There are no infinite $A \subset \omega$ and $n \in \omega$ such that $\forall i \in A \[n(i) = n\]$.
 \end{Claim}
 \begin{proof}
 Suppose not.
 Fix $i, j \in A$ such that $s(i) < s(j)$.
 Then since $c(\{s(i), s(j)\}) = 1$, ${f}_{s(j)}\left(\Delta\left({f}_{s(i)}, {f}_{s(j)}\right)\right) = n(i) = n$.
 As ${f}_{s(j)}$ is strictly increasing, $\Delta\left({f}_{s(i)}, {f}_{s(j)}\right) \leq n$.
 So it is possible to find $i, j, k \in A$  and $m \leq n$ such that $s(i) < s(j) < s(k)$ and
 \begin{align*}
 \Delta\left({f}_{s(i)}, {f}_{s(j)}\right) = \Delta\left({f}_{s(i)}, {f}_{s(k)}\right) = \Delta\left({f}_{s(j)}, {f}_{s(k)}\right) = m.
\end{align*}
However, we now have ${f}_{s(i)} \restrict m = {f}_{s(j)} \restrict m = {f}_{s(k)} \restrict m$, and also that ${f}_{s(j)}(m) = {f}_{s(k)}(m) = n$, which is impossible.
\end{proof}
Let $\cl(B)$ denote the closure (with respect to the usual topology on $\BS$) of $\left\{{f}_{s(i)}: i \in \omega \right\}$.
As $\cl(B)$ is a non-empty closed subset of $\BS$, fix a non-empty pruned subtree $T \subset {\omega}^{< \omega}$ such that $\[T\] = \cl(B)$ (refer to Section \ref{sec:notation} for our notation for subtrees of ${\omega}^{< \omega}$).
Let $\sigma \in T$.
Suppose for a moment that $\existsinf n \in \omega \[{\sigma}^{\frown}{\langle n \rangle} \in T \]$.
Let ${\succc}_{T}(\sigma)$ denote $\{n \in \omega: {\sigma}^{\frown}{\langle n \rangle} \in T\}$ and let $l$ denote $\lc \sigma \rc$.
For each $n \in {\succc}_{T}(\sigma)$, choose $ {i}_{n} \in \omega$ such that ${\sigma}^{\frown}{\langle n \rangle} \subset {f}_{s({i}_{n})}$.
As $B$ is well ordered, by Ramsey's theorem there is $N \in {\[{\succc}_{T}(\sigma)\]}^{\omega}$ such that $\forall m, n \in N\[m < n \implies s({i}_{m}) < s({i}_{n})\]$.
Fix $m, n \in N$ such that $\max\left\{m, n({i}_{m})\right\} <  n$.
However, $\Delta\left( {f}_{s({i}_{m})}, {f}_{s({i}_{n})}\right) = l$ and ${f}_{s({i}_{n})}(l) = n > n({i}_{m})$, contradicting $c(\left\{ s({i}_{m}), s({i}_{n})\right\}) = 1$.
So we conclude that $T$ is finitely branching, and that $\cl(B)$ is a compact subset of $\BS$.

Now, $\left\{{f}_{s(i)}: i \in \omega \right\}$ is an infinite subset of $\cl(B)$.
Let $f \in \cl(B)$ be a complete accumulation point of $\left\{{f}_{s(i)}: i \in \omega \right\}$.
By applying Claim \ref{claim:coloring2} and Ramsey's theorem, it is possible to choose $A \in \cube$ such that for each $i \in A$, $f \neq {f}_{s(i)}$, and for each $i, j \in A$, if $i < j$, then $\Delta \left(f, {f}_{s(i)} \right) < \Delta \left(f, {f}_{s(j)} \right)$ and $n(i) < n(j)$.

Now, since ${f}_{s(\omega + 1)} \neq {f}_{s(\omega)}$, there must be $k \in \{\omega, \omega + 1\}$ such that ${f}_{s(k)} \neq f$.
Fix such a $k \in \{\omega, \omega + 1\}$.
Let $l = {f}_{s(k)}\left(\Delta\left( {f}_{s(k)}, f \right)\right)$.
Choose $i \in A$ such that $\Delta\left( f, {f}_{s(i)}\right) > \Delta\left( {f}_{s(k)}, f \right)$ and $n(i) > l$.
Then $\Delta\left({f}_{s(k)}, {f}_{s(i)}\right) = \Delta\left({f}_{s(k)}, f\right)$.
However, ${f}_{s(k)}\left(\Delta\left({f}_{s(k)}, {f}_{s(i)}\right)\right) = l < n(i)$, contradicting $c(\{s(i), s(k)\}) = 1$.
This contradiction finishes the proof.
\end{proof}
Observe that properties of $\Sa$ such as coherence are not used in the proof of Theorem \ref{thm:coloring}; we only need that every element of $\Sa$ has infinitely many immediate successors in $\Sa$.

It is also worth pointing out the following corollary of the proof of Theorem \ref{thm:coloring} as it shows that while in the model of $\PFA(\Sa)$ the partition relation
${\omega}_{1} \rightarrow~{({\omega}_{1}, \omega + 2)}^{2}$ fails, forcing with the coherent Suslin tree $\Sa$ recuperates it.
\begin{Theorem}
If there is a Suslin tree then ${\omega}_{1} ~\not \rightarrow~{({\omega}_{1}, \omega + 2)}^{2}.$
\end{Theorem}
Although Theorem \ref{thm:main} shows that the statement that ${\omega}_{1} \rightarrow {({\omega}_{1}, \alpha)}^{2}$ for all $\alpha < {\omega}_{1}$ is not equivalent to $\p > {\omega}_{1}$ over $\ZFC + \PID$, it does not give much further information.
\begin{prob} \label{p:partition}
 Find a cardinal invariant $\x$ so that the statement that ${\omega}_{1}~\rightarrow~{({\omega}_{1}, \alpha)}^{2}$ for all $\alpha < {\omega}_{1}$ is equivalent to $\x > {\omega}_{1}$ over $\PID$.
\end{prob}
\noindent The following result that comes from Theorem \ref{beomegaone} and the proof of Theorem \ref{omegaone} above gives a partial answer to this problem.
\begin{Theorem}[Todorcevic]
Assume $\PID$. The following are equivalent.
\begin{enumerate}
\item[(1)] $\b > {\omega}_{1}$.
\item[(2)] ${\omega}_{1}\rightarrow~{({\omega}_{1}, \omega + 2)}^{2}$.
\end{enumerate}
\end{Theorem}
\noindent This leads us to the following version of Problem \ref{p:partition}.
\begin{prob} \label{p:partition2}
 Is the statement that ${\omega}_{1}~\rightarrow~{({\omega}_{1}, \alpha)}^{2}$ for all $\alpha < {\omega}_{1}$ equivalent to $\mathfrak{b} > {\omega}_{1}$ over $\PID$?
\end{prob}
\noindent Note that if, under $\PID,$ the cardinal invariant inequality $\b > {\omega}_{1}$ does not correspond to ${\omega}_{1}\rightarrow~{({\omega}_{1}, \alpha)}^{2}$, there must be a minimal $\alpha < {\omega}_{1}$ with this property. The proof of Theorem \ref{omegaone} actually shows that such $\alpha$ is quite large.
\section{$\PID$ and five cofinal types} \label{sec:tukey}
Recall that by Theorem \ref{thm:five} above, under $\PID + \p > {\omega}_{1}$, $1$, $\omega$, ${\omega}_{1}$, $\omega \times {\omega}_{1}$, and ${\[{\omega}_{1}\]}^{< \omega}$ are the only cofinal types of directed sets of size at most ${\aleph}_{1}$.
In this section we find a cardinal invariant which captures, in the sense described in the introduction, this statement.
This cardinal invariant is not one of the naturally occurring ones.
Rather it is the minimum of two cardinal invariants, the well-known bounding number $\b$ and another cardinal (see Definition \ref{def:newcardinal} below) which has not been investigated as throughly.
\begin{Def} \label{def:pseudo}
 Let $\langle D, \leq \rangle$ be a directed set and suppose $X \subset D$.
 We say that $X$ is \emph{pseudobounded} if $\forall A \in {\[X\]}^{\omega} \exists B \in {\[A\]}^{\omega}\[B \ \text{is bounded in} \ D\]$.
 $D$ is said to be \emph{$\sigma$-pseudobounded} if $D = {\bigcup}_{n \in
\omega}{{X}_{n}}$, where for each $n \in \omega$, ${X}_{n}$ is pseudobounded in
$D$.
\end{Def}
\begin{Lemma} \label{lem:idealnotpseudo}
 Let $\I$ be a tall ideal on $\omega$ such that $\langle \I, \subset \rangle$ is
$\sigma$-pseudobounded.
 Then ${\[{\omega}_{1}\]}^{< \omega} {\not\leq}_{T} \langle \I, \subset \rangle$
and $\langle \I, \subset \rangle \; {\not\leq}_{T} \; \omega \times
{\omega}_{1}$.
\end{Lemma}
\begin{proof}
 It is clear that ${\[{\omega}_{1}\]}^{< \omega} \; {\not\leq}_{T} \; \langle
\I, \subset \rangle$ because $\I$ is $\sigma$-pseudobounded.
 Suppose for a contradiction that $\langle \I, \subset \rangle \; {\leq}_{T} \;
\omega \times {\omega}_{1}$.
 Then there exists $\{X(n, \alpha): n < \omega \wedge \alpha < {\omega}_{1}\}
\subset \I$ such that for any $A \in \cube$ and $\{{X}_{n}: n \in A\} \subset
{\[{\omega}_{1}\]}^{{\omega}_{1}}$, $\{X(n, \alpha): n \in A \wedge \alpha \in
{X}_{n}\}$ is cofinal in $\I$.
 First of all since $\I$ is proper, for each $n \in \omega$, there must be
${k}_{n} \in \omega$ and ${X}_{n} \in {\[{\omega}_{1}\]}^{{\omega}_{1}}$ such that
$\forall \alpha \in {X}_{n}\[{k}_{n} \notin X(n, \alpha)\]$.
 Since $\I$ is tall there is $A \in \cube$ such that $\{{k}_{n}: n \in A \} \in
\I$ (either $\{{k}_{n}: n \in \omega\}$ is finite, in which case $A = \omega$
works, or else use tallness).
 Now $\{X(n, \alpha): n \in A \wedge \alpha \in {X}_{n}\}$ is cofinal in $\I$.
 So there is $n \in A$ and $\alpha \in {X}_{n}$ such that $\{{k}_{m}: m \in A\} \subset X(n, \alpha)$. However, ${k}_{n} \notin X(n, \alpha)$.
\end{proof}
For any directed set $D$, and $X$ a directed cofinal subset of $D$, $D \;
{\equiv}_{T} \; X$.
Therefore, if $\I$ is a tall $\sigma$-pseudobounded ideal on $\omega$ and $X
\subset \I$ is cofinal, directed, and has size at most ${\omega}_{1}$, then
$\langle X, \subset \rangle$ is not Tukey equivalent to any of $1$, $\omega$,
${\omega}_{1}$, $\omega \times {\omega}_{1}$, ${\[{\omega}_{1}\]}^{< \omega}$.
\begin{Def} \label{def:newcardinal}
 $\cof({\F}_{\sigma})$ is the least $\kappa$ such that there exists a tall,
$\sigma$-pseudobounded ${F}_{\sigma}$ ideal $\I$ on $\omega$ and a directed
cofinal $X \subset \I$ such that $\lc X \rc = \kappa$.
\end{Def}
It is clear that ${\omega}_{1} \leq \cof({\F}_{\sigma}) \leq \c$.
It is also easy to see that $\cov(\M) \leq \cof({\F}_{\sigma})$.
Later in this section we will prove that $\b$ and $\cof({\F}_{\sigma})$ are
independent, even assuming $\PID$.
We do not know whether the same cardinal invariant is obtained if the requirement that $\I$ be $\sigma$-pseudobounded is dropped from the definition of $\cof({\F}_{\sigma})$.
This is closely related to the well-known question of whether every tall ${F}_{\sigma}$ ideal on $\omega$ is either $\sigma$-pseudobounded or Tukey equivalent to ${\[\c\]}^{< \omega}$.
\begin{conj} \label{conj:1}
 Let ${\cof}^{\ast}({\F}_{\sigma})$ be the least $\kappa$ such that there exists a tall ${F}_{\sigma}$ ideal $\I$ on $\omega$ and a directed cofinal $X \subset \I$ such that $\lc X \rc = \kappa$.
 Then $\cof({\F}_{\sigma}) = {\cof}^{\ast}({\F}_{\sigma})$.
\end{conj}
\begin{Theorem} \label{thm:5yes}
Assume $\PID$.
The following are equivalent.
  \begin{enumerate}
   \item
    $\min\{\b, \cof({\F}_{\sigma})\} > {\omega}_{1}$.
   \item
    $1$, $\omega$, ${\omega}_{1}$, $\omega \times {\omega}_{1}$, and
${\[{\omega}_{1}\]}^{< \omega}$ are the only cofinal types of directed sets of
size at most ${\aleph}_{1}$.
  \end{enumerate}
\end{Theorem}
\begin{proof}
By the results of Todorcevic in \cite{partition} it follows that if $\b =
{\omega}_{1}$, then there is a directed set of size ${\omega}_{1}$ whose cofinal
type is different from any of $1$, $\omega$, ${\omega}_{1}$, $\omega \times
{\omega}_{1}$, and ${\[{\omega}_{1}\]}^{< \omega}$.

Next, suppose that $\cof({\F}_{\sigma}) = {\omega}_{1}$.
Let $\I$ and $X \subset \I$ witness this.
Then $\lc X \rc = {\omega}_{1}$, and by Lemma \ref{lem:idealnotpseudo}, the
cofinal type of $\langle X, \subset \rangle$ is not one of $1$, $\omega$,
${\omega}_{1}$, $\omega \times {\omega}_{1}$, and ${\[{\omega}_{1}\]}^{<
\omega}$. This proves $\neg (1) \implies \neg (2)$.

For the other direction, we assume (1) and prove (2).
It is easy to see (for example see \cite{isbell}) that if $D$ is a directed set
of size at most ${\aleph}_{1}$ and $D \; {\not\geq}_{T} \; {\omega}_{1} \times \omega$, then $D \; {\equiv}_{T} \; 1$, or $D \; {\equiv}_{T} \; \omega$, or $D \; {\equiv}_{T} \; {\omega}_{1}$.
Therefore, fixing a directed set $D$ with $\lc D \rc \leq {\aleph}_{1}$, it is
sufficient to show that either $D$ contains an uncountable set $X$ all of whose
infinite subsets are unbounded (in which case ${\[{\omega}_{1}\]}^{< \omega} \;
{\equiv}_{T} \; D$) or else that $D = {\bigcup}_{n \in \omega}{{X}_{n}}$ where
for each $n \in \omega$ and each $A \in {\[{X}_{n}\]}^{\omega}$, $A$ is bounded
in $D$ (in which case $D \; {\leq}_{T} \; {\omega}_{1} \times \omega$).
The proof of this proceeds in two cases.
We first make some preliminary remarks.
For $x \in D$, $\pred(x)$ denotes $\{y \in D: y \leq x\}$.
Let $A \subset D$.
Define the \emph{trace of $D$ on $A$}, $\tr(D, A) = \{B \subset A: \exists x \in
D \[B \subset \pred(x)\]\}$.
Note that since $D$ is directed, $\tr(D, A)$ is an ideal on $A$.

Case I: For each $A \in {\[D\]}^{\leq \omega}$, $\tr(D, A)$ is not tall.
Let $\I = \{A \in {\[D\]}^{\leq \omega}: \forall x \in D \[ \lc A \cap \pred(x)
\rc < \omega \]\}$.
It is easy to see that $\I$ is an ideal.
To see that it is a P-ideal, fix $\{{A}_{n}: n \in \omega \} \subset \I$.
Without loss of generality, the ${A}_{n}$ are pairwise disjoint and infinite.
For each $x \in D$ and $n \in \omega$, put $\pred(x, n) = \pred(x) \cap
{A}_{n}$; this is a finite subset of ${A}_{n}$.
As $\lc D \rc \leq {\omega}_{1} < \b$, it is possible to find ${F}_{n} \subset
{A}_{n}$ such that
\begin{align*}
 \forall x \in D \forallbutfin n \in \omega \[\pred(x, n) \subset {F}_{n}\].
\end{align*}
Putting $A = {\bigcup}_{n \in \omega}{\left( {A}_{n} \setminus {F}_{n}
\right)}$, it is clear that $A \in \I$ and that $\forall n \in \omega \[{A}_{n} \; {\subset}^{\ast} \; A\]$.
Now, if there is an uncountable $X \subset D$ such that ${\[X\]}^{\leq \omega}
\subset \I$, then it is clear that every infinite subset of $X$ is unbounded in
$D$, and therefore $D \; {\equiv}_{T} \; {\[{\omega}_{1}\]}^{< \omega}$.
So suppose that there exist $\{{X}_{n}: n \in \omega\}$ which are pairwise
disjoint such that $D = {\bigcup}_{n \in \omega}{{X}_{n}}$, and $\forall n \in \omega \[{\[{X}_{n}\]}^{\omega} \cap \I = 0\]$.
We claim that for each $n \in \omega$, and each $A \in {\[{X}_{n}\]}^{\omega}$,
$A$ is bounded in $D$.
This is sufficient to show that $D \; {\leq}_{T} \; {\omega}_{1} \times \omega$.
Fix $n \in \omega$ and $A \in {\[{X}_{n}\]}^{\omega}$.
The hypothesis of Case I implies that either $\tr(D, A)$ is not a proper ideal
on $A$ or that there exists $C \in {\[A\]}^{\omega}$ such that $\forall B \in \tr(D, A)\[\lc C \cap B \rc < \omega\]$.
Suppose for a moment that $\exists C \in {\[A\]}^{\omega}\forall B \in \tr(D, A)\[\lc C \cap B \rc < \omega\]$.
Then for any $x \in D$, $\pred(x) \cap C$ is finite, whence $C \in \I$.
However, this means that $C \in \I \cap {\[{X}_{n}\]}^{\omega}$, contradicting
$\I \cap {\[{X}_{n}\]}^{\omega} = 0$.
Therefore, it must be the case that $\tr(D, A)$ is not a proper ideal on $A$.
So $A \in \tr(D, A)$, whence $A$ is bounded.
This completes Case I.

Case II: There exists $A \in {\[D\]}^{\omega}$ for which $\J = \tr(D, A)$ is a
tall ideal on $A$.
In particular, $\J$ is a proper ideal on $A$.
Identifying $A$ with a copy of $\omega$, it makes sense to talk about the
descriptive complexity of $\J$.
Put $X = \{\pred(x) \cap A: x \in D\} \subset \J$, and note that $X$ is a
cofinal subset of $\langle \J, \subset \rangle$ of size at most ${\omega}_{1}$.
Moreover, $X$ is directed.
To see this if $x, y \in D$, then choosing $z \in D$ such that $x, y \leq z$, it
is clear that $\pred(z) \cap A \in X$, and that $\pred(x) \cap A \subset \pred(z) \cap A$ and $\pred(y) \cap A \subset \pred(z) \cap A$.
Define $\I = \{\F \in {\[\J\]}^{\leq \omega}: \forall B \in \J \[\lc \F \cap \Pset(B)\rc < \omega \]\}$.
Using the fact that $\J$ has a cofinal subset of size at most ${\omega}_{1}$ and
the hypothesis that $\b > {\omega}_{1}$, it is easy to check that $\I$ is a
P-ideal.
First suppose there is an uncountable $\GG \subset \J$ such that
${\[\GG\]}^{\leq \omega} \subset \I$.
For each $B \in \GG$ choose ${x}_{B} \in D$ such that $B \subset \pred({x}_{B})$.
Let $\HH \subset \GG$ be infinite (not necessarily countable).
We claim that $\{{x}_{B}: B \in \HH\}$ is unbounded in $D$.
For if not, then fix $x \in D$ such that $x \geq {x}_{B}$, for all $B \in \HH$,
and note that $\pred(x) \cap A \in \J$.
Now for any $B \in \HH$, $B \subset \pred({x}_{B}) \subset \pred(x)$, and so $B \in \Pset(\pred(x) \cap A)$.
Thus if $\LLL \in {\[\HH\]}^{\omega}$, then $\LLL \cap \Pset(\pred(x) \cap A)$
is infinite, whence $\LLL \notin \I$, contradicting ${\[\GG\]}^{\leq \omega} \subset \I$.
It follows that $\{{x}_{B}: B \in \GG\}$ is an uncountable subset of $D$, and
that no infinite subset of it is bounded in $D$.

Next, suppose that there exist $\{{\GG}_{n}: n \in \omega\}$ which are pairwise
disjoint such that $\J = {\bigcup}_{n \in \omega}{{\GG}_{n}}$ and $\forall n \in \omega \[{\[{\GG}_{n}\]}^{\omega} \cap \I = 0\]$.
We first claim that each ${\GG}_{n}$ is pseudobounded in $\J$, so that $\J$ is
$\sigma$-pseudobounded.
Fix $n \in \omega$ and let $\F \in {\[{\GG}_{n}\]}^{\omega}$.
Since $\F \notin \I$ and $\F \in {\[\J\]}^{\leq \omega}$, there exists $B \in
\J$ such that $\F \cap \Pset(B)$ is infinite.
It is clear that $\LLL = \F \cap \Pset(B) \in {\[\F\]}^{\omega}$ and that for
all $C \in \LLL$, $C \subset B$, whence $\LLL$ is bounded in $\J$.
It now follows from the assumption that $\cof({\F}_{\sigma}) > {\omega}_{1}$
that $\J$ is not an ${\F}_{\sigma}$ ideal (with respect to the natural topology
on $\Pset(A)$).
For each $n \in \omega$, let ${\HH}_{n}$ be the closure of ${\GG}_{n}$ with
respect to the usual topology on $\Pset(A)$.
So there is $n \in \omega$ such that ${\HH}_{n} \not\subset \J$.
Fix such $n$ and fix $C \in {\HH}_{n} \setminus \J$.
Let $\{{a}_{i}: i \in \omega\}$ enumerate $A$.
Choose $\{{B}_{m}: m \in \omega\} \subset {\GG}_{n}$ such that for each $m \in
\omega$, ${B}_{m} \cap \{{a}_{i}: i < m \} = C \cap \{{a}_{i}: i < m \}$ and
moreover $\forall i < m \[{B}_{i} \neq {B}_{m}\]$.
Thus $\{{B}_{m}: m \in \omega\} \in {\[{\GG}_{n}\]}^{\omega}$, and so there
exists $B \in \J$ for which $\existsinf m \in \omega \[{B}_{m} \subset B \]$.
We claim that this implies that $C \subset B$.
Indeed if ${a}_{i} \in C$ for some $i \in \omega$, then choose $m > i$ such that
${B}_{m} \subset B$.
Then it is clear that ${a}_{i} \in {B}_{m} \subset B$.
Thus it follows that $C \in \J$, a contradiction.
Since we have a contradiction from the second alternative of $\PID$, it must be
that in Case II the first alternative always occurs.
This finishes the proof.
\end{proof}
A noteworthy feature of this result is that the cardinal $\cof({\F}_{\sigma})$ speaks about the cofinal structure of definable ideals of size continuum while (2) of Theorem \ref{thm:5yes} is part of the general theory of cofinal types.
\begin{Cor} \label{cor:suslinfive}
 $\PFA(\Sa)$ implies that the coherent Suslin tree $\Sa$ forces that $1$,
$\omega$, ${\omega}_{1}$, $\omega \times {\omega}_{1}$, and
${\[{\omega}_{1}\]}^{< \omega}$ are the only cofinal types of directed sets of
size at most ${\aleph}_{1}$.
\end{Cor}
\begin{proof}
 It is a theorem of Todorcevic~\cite{todcoh} that if $\V$ satisfies $\PFA(\Sa)$ and if $G$ is $(\V, \Sa)$-generic, then in $\V\[G\]$, both $\PID$ and $\b > {\omega}_{1}$ hold.
 So in view of Theorem \ref{thm:5yes}, it suffices to check that in $\V\[G\]$,
$\cof(\J) > {\omega}_{1}$ for any tall ${F}_{\sigma}$ ideal $\J$.
 As $\Sa$ does not add any reals, all ${F}_{\sigma}$ ideals in $\V\[G\]$ are
coded in $\V$.
 So fix $\J \in \V$ a code for an ${F}_{\sigma}$ ideal, and suppose $\forces
``\J \ \text{is tall}''$.
 Let $\{{\mathring{x}}_{\alpha}: \alpha < {\omega}_{1}\} \subset {\V}^{\Sa}$
with $\forces ``{\mathring{x}}_{\alpha} \in \J''$, for each $\alpha <
{\omega}_{1}$.
 We claim that $\forces {\{{\mathring{x}}_{\alpha}: \alpha < {\omega}_{1}\} \ \text{is not cofinal in} \ \J}$.
 Fix $p \in \Sa$.
 For each $q \leq p$ and $\alpha < {\omega}_{1}$ find $x(\alpha, q) \in \J$ and
$r(\alpha, q) \leq q$ such that $r(\alpha, q) \; \forces \;
{\mathring{x}}_{\alpha} = x(\alpha, q)$.
 Since $\p > {\omega}_{1}$ holds in $\V$ and since $\J$ is a proper ideal, find
$x \in \cube$ such that $\forall \alpha < {\omega}_{1} \forall q \leq p \[x \;
{\subset}^{\ast} \; \left( \omega \setminus x(\alpha, q) \right)\]$.
 Since $\forces ``\J \ \text{is tall}''$, there is $y \in \J \cap
{\[x\]}^{\omega}$.
 Now it is clear that for each $\alpha < {\omega}_{1}$ $p \; \forces \; y \; \not\subset \; {\mathring{x}}_{\alpha}$.
\end{proof}
\begin{Cor} \label{cor:add}
$\PID + \add(\M) > {\omega}_{1}$ implies that $1$, $\omega$, ${\omega}_{1}$,
$\omega \times {\omega}_{1}$, and ${\[{\omega}_{1}\]}^{< \omega}$ are the only
cofinal types of directed sets of size at most ${\aleph}_{1}$.
\end{Cor}
Next, we work towards showing that $\b$ and $\cof({\F}_{\sigma})$ are mutually
independent even in the presence of $\PID$
\begin{Def} \label{def:poly}
A sequence $I = \langle {k}_{n}: n \in \omega \rangle$ is called an
\emph{interval partition} if ${k}_{0} = 0$ and $\forall n \in \omega\[{k}_{n} <
{k}_{n + 1}\]$.
For an interval partition $I$ and $n \in \omega$, ${I}_{n} = [{k}_{n}, {k}_{n +
1})$.
Let $I$ be an interval partition such that $\forall n \in \omega \[\lc {I}_{n} \rc \geq {2}^{n}\]$.
In the rest of this section, \emph{polynomial} means a polynomial with integer
co-efficients.
Define ${\I}_{\mathrm{poly}}(I)$ to be $\{A \subset \omega: \text{there is a
polynomial} \ p(n) \ \text{such that} \ \forall n \in \omega \[\lc {I}_{n} \cap A \rc \leq p(n) \]\}$.
It is clear that this an ${F}_{\sigma}$ ideal on $\omega$.
\end{Def}
\begin{Def} \label{def:laverproperty}
 A poset $\P$ is said to have the \emph{Laver property} if for each sequence of
finite sets $\langle H(n): n \in \omega \rangle$ in the ground model $\V$, for
each $(\V, \P)$-generic $G$, and for each $f \in \V[G] \cap {\prod}_{n \in
\omega}{H(n)}$, there is a $K \in \V \cap {\prod}_{n \in \omega}{\Pset(H(n))}$
such that $\forall n \in \omega \[\lc K(n) \rc \leq n + 1 \wedge f(n) \in
K(n)\]$.
\end{Def}
The following is a special case of a result of Hru{\v{s}}{\'a}k, Rojas-Rebolledo, and Zapletal\cite{coffsigma}.
\begin{Lemma}[Hru{\v{s}}{\'a}k, Rojas-Rebolledo, and Zapletal]\label{lem:laver}
Fix $I \in \V$.
Let $\P$ be any poset with the Laver property.
Then ${\forces}{``\V \cap {\I}_{\mathrm{poly}}(I) \ \text{is a cofinal subset
of} \ {\I}_{\mathrm{poly}}(I)''}$.
\end{Lemma}
\begin{proof}
 In the ground model $\V$, define for each polynomial $p$ and $n \in \omega$, the set $H(p, n) = \{s \subset {I}_{n}: \lc s \rc \leq p(n)\}$.
 Let $G$ be $(\V, \P)$ generic and let $A \in \V\[G\] \cap
{\I}_{\mathrm{poly}}(I)$.
 Let $p$ be a polynomial witnessing this.
 Define ${F}_{A} \in {\prod}_{n \in \omega}{H(p, n)}$ by ${F}_{A}(n) = A \cap {I}_{n}$.
 By the Laver property, find $K \in \V \cap {\prod}_{n \in \omega}{\Pset(H(p, n))}$ such that $\forall n \in \omega \[\lc K(n) \rc \leq n + 1 \wedge
{F}_{A}(n) \in K(n)\]$.
 Working in $\V$ define ${s}_{n} = \bigcup K(n)$, for each $n \in \omega$.
 Clearly, ${s}_{n} \subset {I}_{n}$ and $\lc {s}_{n} \rc \leq (n + 1)p(n)$.
 Therefore, $B = {\bigcup}_{n \in \omega}{{s}_{n}} \in \V \cap
{\I}_{\mathrm{poly}}(I)$ and $A \subset B$.
\end{proof}
\begin{Lemma} \label{lem:nottop}
 For any $I$ there exist $\{{X}_{n}: n \in \omega\}$ such that
${\I}_{\mathrm{poly}}(I) = {\bigcup}_{n \in \omega}{{X}_{n}}$ and for each $n
\in \omega$, every infinite subset of ${X}_{n}$ has a further infinite subset
that is bounded in ${\I}_{\mathrm{poly}}(I)$. In other words, ${\I}_{\mathrm{poly}}(I)$ is $\sigma$-pseudobounded.
\end{Lemma}
\begin{proof}
 Let $\{{p}_{n}: n \in \omega\}$ enumerate all polynomials.
 Define ${X}_{n} = \{A \in {\I}_{\mathrm{poly}}(I): \forall m \in \omega \[\lc
{I}_{m} \cap A \rc \leq {p}_{n}(m)\]\}$.
 It is clear that ${\I}_{\mathrm{poly}}(I) = {\bigcup}_{n \in \omega}{{X}_{n}}$
and that each ${X}_{n}$ is a closed subset of $\Pset(\omega)$.
 Now fix $n$ and an infinite $Y \subset {X}_{n}$.
 Let $A \in {X}_{n}$ be a complete accumulation point of $Y$.
 For each $m \in \omega$ choose ${B}_{m} \in Y$ such that ${B}_{m} \cap \left(
{\bigcup}_{i \leq m}{{I}_{i}} \right) = A \cap \left( {\bigcup}_{i \leq
m}{{I}_{i}} \right)$ and moreover $\forall i < m \[{B}_{m} \neq {B}_{i} \]$.
 Thus $\{{B}_{m}: m \in \omega\} \in {\[Y\]}^{\omega}$.
 For each $i \in \omega$, put ${s}_{i} = {\bigcup}_{m \in \omega}{\left( {B}_{m} \cap {I}_{i}\right)}$.
 It is clear that ${s}_{i} = \left( {\bigcup}_{m < i}{\left( {B}_{m} \cap
{I}_{i}\right)} \right) \cup \left( A \cap {I}_{i}\right)$.
 Therefore $\lc {s}_{i} \rc \leq (i + 1){p}_{n}(i)$ and ${\bigcup}_{m \in
\omega}{{B}_{m}} \in {\I}_{\mathrm{poly}}(I)$, as needed.
\end{proof}
\begin{Theorem} \label{thm:bnotenough}
 Assume that there is a supercompact cardinal $\kappa$. Then $\PID + \b >
{\omega}_{1}$ does not imply $\cof({\F}_{\sigma}) > {\omega}_{1}$. Also there is a model of $\PID + \cof({\F}_{\sigma}) > {\omega}_{1}$, where $\d = {\omega}_{1}$ (\emph{a fortiori} $\b = {\omega}_{1}$).
\end{Theorem}
\begin{proof}
By results of \cite{PID} and \cite{GPID}, if $\I$ is any P-ideal, then there is
a proper poset not adding reals, call it ${\P}_{\I}$, that forces $\PID$ with
respect to $\I$.
Using a Laver diamond in a ground model satisfying $\CH$, do a CS iteration $\langle {\P}_{\alpha}, {\mathring{\Q}}_{\alpha}: \alpha \leq \kappa \rangle$ as follows.
Given ${\P}_{\alpha}$, if the Laver diamond picks a ${\P}_{\alpha}$ name for a
P-ideal $\mathring{\I}$, then let ${\mathring{\Q}}_{\alpha}$ be a full
${\P}_{\alpha}$ name such that ${\forces}_{\alpha} {\mathring{\Q}}_{\alpha} =
{\P}_{\mathring{\I}}$.
Else let  ${\mathring{\Q}}_{\alpha}$ be a full ${\P}_{\alpha}$ name such that
${\forces}_{\alpha} {\mathring{\Q}}_{\alpha} \ \text{is Laver forcing}$.
Note that cofinally often we will have ${\forces}_{\alpha}
{\mathring{\Q}}_{\alpha} \ \text{is Laver forcing}$.
Also, note that each iterand is forced to have the Laver property, which is
preserved in CS iterations.
So if $G$ is $(\V, {\P}_{\kappa})$ generic, then in $\V\[G\]$, $\PID + \b >
{\omega}_{1}$ holds.
Moreover, if $I \in \V$, then by Lemma~\ref{lem:laver} $\V \cap
{\I}_{\mathrm{poly}}(I)$ is cofinal in $\VG \cap {\I}_{\mathrm{poly}}(I)$.
By Lemma \ref{lem:nottop}, in $\VG$, ${\I}_{\mathrm{poly}}(I)$ is a tall, $\sigma$-pseudobounded, ${F}_{\sigma}$ ideal, and $\V \cap {\I}_{\mathrm{poly}}(I)$ is a directed cofinal subset of $\VG \cap {\I}_{\mathrm{poly}}(I)$ of size ${\omega}_{1}$.
So it witnesses $\cof({\F}_{\sigma}) = {\omega}_{1}$.
This finishes the proof of the first statement.

For the second statement, we use the well-known result of Laflamme~\cite{zapping} that for any ground model $\V$ and any ${F}_{\sigma}$ ideal $\I$ on $\omega$ belonging to $\V$, there is a proper $\BS$-bounding poset ${\Q}_{\I} \in \V$ such that ${\Q}_{\I}$ adds an infinite subset of $\omega$ that is almost disjoint from every member of $\V \cap \I$.
Once again, fix a ground model satisfying $\CH$ and a Laver diamond in that ground model.
Do a CS iteration $\langle {\P}_{\alpha}, {\mathring{\Q}}_{\alpha}: \alpha \leq \kappa \rangle$ as follows.
Given ${\P}_{\alpha}$, if the Laver diamond picks a ${\P}_{\alpha}$ name for a P-ideal $\mathring{\I}$, then let ${\mathring{\Q}}_{\alpha}$ be a full ${\P}_{\alpha}$ name such that ${\forces}_{\alpha} \; {\mathring{\Q}}_{\alpha} = {\P}_{\mathring{\I}}$.
If the Laver diamond picks a pair of ${\P}_{\alpha}$ names $\langle \mathring{\I}, \mathring{X} \rangle$ such that ${\forces}_{\alpha} {``\mathring{\I} \ \text{is a tall} \ {F}_{\sigma} \ \text{ideal}''}$ and ${\forces}_{\alpha} {``\mathring{X} \subset \mathring{\I} \ \text{and} \ \lc \mathring{X} \rc < \kappa''}$, then let ${\mathring{\Q}}_{\alpha}$ be a full ${\P}_{\alpha}$ name for ${\Q}_{\mathring{\I}}$.
If neither of these happens, then ${\mathring{\Q}}_{\alpha}$ is a full ${\P}_{\alpha}$ name for the trivial poset.
Note that since each iterand is proper and $\BS$-bounding, ${\P}_{\kappa}$ is $\BS$-bounding.
Therefore, if $G$ is $(\V, {\P}_{\kappa})$ generic, then $\PID + \d = {\omega}_{1}$ holds in $\VG$.
Also in $\VG$, if $\I$ is a tall ${F}_{\sigma}$ ideal and $X \subset \I$ is of size at most ${\omega}_{1}$, then there exists $a \in \cube$ such that $\forall x \in X \[\lc x \cap a \rc < \omega\]$.
This implies that $X$ is not cofinal in $\I$.
Therefore, $\cof({\F}_{\sigma}) > {\omega}_{1}$ holds in $\VG$.
\end{proof}
\bibliographystyle{amsplain}
\bibliography{Bibliography}
\end{document}